\documentclass[a4paper,12pt]{article}
\usepackage{amsmath,amssymb,amsthm}
\usepackage{amscd}

\makeatletter

\@addtoreset{equation}{subsection}
\makeatother

\newtheorem{dfn}{Definition}[subsection]
\newtheorem{prop}[dfn]{Proposition}
\newtheorem{thm}[dfn]{Theorem}
\newtheorem{lem}[dfn]{Lemma}
\newtheorem{cor}[dfn]{Corollary}

\newtheorem{conj}[dfn]{Conjecture}

\newcommand\ag{\mathfrak{a}_g}
\newcommand\agminus{\mathfrak{a}_g^-}
\newcommand\Der{{\rm Der}}
\newcommand\deromega{{\rm Der}_\omega}
\newcommand\Hom{{\rm Hom}}
\newcommand\hotimes{\hat{\otimes}}
\newcommand\Ker{{\rm Ker}}

\newcommand\Qpi{\mathbb{Q}\hat{\pi}}

\newcommand\fsp{\mathfrak{sp}}
\newcommand\T{\widehat{T}}

\begin{document}

\title
{The center of the Goldman Lie algebra
of a surface of infinite genus}
\author{Nariya Kawazumi and Yusuke Kuno}
\date{}
\maketitle

\begin{abstract}
Let $\Sigma_{\infty, 1}$ be the inductive limit of compact oriented
surfaces with one boundary component. We prove the center of the 
Goldman Lie algebra of the surface $\Sigma_{\infty,1}$ is spanned by 
the constant loop. A similar statement for a closed oriented surface was
conjectured by  Chas and Sullivan, and proved by Etingof. Our result is
deduced from a computation of the center of the Lie algebra of oriented
chord diagrams. 
\end{abstract}

\section{Introduction}
Let $S$ be a connected oriented surface and let $\hat{\pi} =
\hat{\pi}(S)=[S^1, S]$ be the set of free homotopy classes of oriented loops on $S$. 
In 1986 Goldman \cite{Go} introduced a Lie algebra structure on the vector 
space $\Qpi$ spanned by the set $\hat{\pi}$. Nowadays this Lie algebra
is called {\it the Goldman Lie algebra}, whose bracket is defined as follows.
Let $\alpha$, $\beta$ be immersed loops on $S$ such that
their intersections consist of transverse double points.
For each $p\in \alpha \cap \beta$, let $|\alpha_p\beta_p|$
be the free homotopy class of the loop first going the oriented
loop $\alpha$ based at $p$, then going $\beta$ based at $p$.
Also let $\varepsilon(p;\alpha,\beta)\in \{ \pm 1\}$ be the local
intersection number of $\alpha$ and $\beta$ at $p$, and set
$$[\alpha,\beta]:=\sum_{p\in \alpha \cap \beta}
\varepsilon(p;\alpha,\beta)|\alpha_p\beta_p|\in \mathbb{Q}\hat{\pi}.$$
He proved this descends to a Lie bracket on the vector space $\mathbb{Q}\hat{\pi}$.
It is clear from the definition that if $\alpha$ and $\beta$ are freely
homotopic to disjoint curves, then $[\alpha,\beta]=0$.
In the same paper, he proved a part of the opposite direction.
\begin{thm}[Goldman \cite{Go} Theorem 5.17]
\label{Goldman}
Let $\alpha, \beta \in \hat{\pi}$, where $\alpha$ is represented 
by a simple closed curve. Then $[\alpha, \beta]= 0$ in $\Qpi$ if and only if 
$\alpha$ and $\beta$ are freely homotopic to disjoint curves.
\end{thm}

It is a fundamental problem to compute the center of a given Lie algebra.
We denote the center of a Lie algebra $\mathfrak{g}$ by 
$Z(\mathfrak{g})$. If $S$ is closed, then, from this theorem,
$\hat{\pi}\cap Z(\Qpi) = \{1\}$. Here $1 \in \hat{\pi}$ is 
the constant loop. 
Chas and Sullivan conjectured the following, and Etingof proved it.
\begin{thm}[Etingof \cite{E}]
\label{Etingof}
If $S$ is closed, the center $Z(\Qpi)$ of the Lie algebra $\Qpi$ is 
spanned by the constant loop $1 \in \hat{\pi}$.
\end{thm}
His proof is based on symplectic geometry of the moduli space 
of flat $GL_N(\mathbb{C})$-bundles  over the surface $S$. 
In this paper we study a variant of the Chas-Sullivan conjecture
and give a supporting evidence for it. The variant, 
in the most general setting, is stated as follows.
\begin{conj}\label{conjecture}
For any connected oriented surface $S$, 
the center $Z(\Qpi)$ is spanned by the set $\hat{\pi}\cap Z(\Qpi)$.
\end{conj}
Let $\Sigma_{g,1}$ be a compact connected oriented surface of 
genus $g$ with one boundary component, $\zeta$ the simple loop 
going around the boundary in the opposite direction. 
Then, for $S=\Sigma_{g,1}$, we have $\hat{\pi}\cap Z(\Qpi)
= \{\zeta^n; n \in \mathbb{Z}\}$ by Theorem \ref{Goldman}. 
Hence Conjecture \ref{conjecture} for $S=\Sigma_{g,1}$ is 
given as follows.
\begin{conj}\label{conj-Sigma}
$$
Z(\Qpi(\Sigma_{g,1})) = \bigoplus_{n\in\mathbb{Z}}\mathbb{Q}\zeta^n.
$$
\end{conj}
This conjecture is still open. We shall study a surface 
of infinite genus, instead.
Gluing a compact connected oriented surface $\Sigma_{1,2}$ 
of genus $1$ with 
$2$ boundary components to the surface $\Sigma_{g,1}$ along 
the boundary, we obtain an embedding $i^g_{g+1}\colon
\Sigma_{g,1}\hookrightarrow 
\Sigma_{g+1,1}$. We define a connected oriented surface 
$\Sigma_{\infty,1}$ as the inductive limit of these embeddings. 
Our main result supports Conjecture \ref{conjecture}.
The conjecture holds for the surface $S=\Sigma_{\infty,1}$: 
\begin{thm}\label{main}
$$
Z(\Qpi(\Sigma_{\infty,1})) = \mathbb{Q}1.
$$
\end{thm}
Our method of proof differs from Etingof's proof of Theorem \ref{Etingof},
and is based on our previous result \cite{KK} Theorem 1.2.1 
which connects the Goldman Lie algebra $\Qpi(\Sigma_{g,1})$ 
to Kontsevich's ``associative'' formal  symplectic geometry $\ag$. 
The notion of a {\it symplectic expansion} introduced by Massuyeau
\cite{Mas} plays a vital role there.
Theorem \ref{main} is deduced
from  a computation of the center of {\it the Lie algebra of oriented 
chord diagrams}, which is introduced in \S3.
This Lie algebra can be thought as the ``limit" of
the $\mathfrak{sp}$-invariants $(\ag)^{\mathfrak{sp}}$, $g\to \infty$,
where $\mathfrak{sp}=\mathfrak{sp}_{2g}(\mathbb{Q})$,
and its bracket is defined by a diagrammatic way.
Along the proof we also prove that a counterpart to Conjecture \ref{conj-Sigma} 
in the formal symplectic geometry side, is true in a stable range
(Theorem \ref{stabilized}).

\vspace{0.3cm}
This paper is organized as follows. In \S2, we recall symplectic
expansions, Kontsevich's ``associative" $\ag$, and our previous result.
In \S3, we give a description of the
$\mathfrak{sp}$-invariants $(\ag)^{\mathfrak{sp}}$ by
labeled chord diagrams. Looking at the bracket on $(\ag)^{\mathfrak{sp}}$,
we arrive at the definition of the Lie algebra of oriented chord diagrams.
We determine the center of this Lie algebra, and compute
the center of the ``associative" $\mathfrak{a}_g^-$,
an extension of $\mathfrak{a}_g$, in a stable range.
This gives a supporting evidence for Conjecture \ref{conj-Sigma},
since it enables us to approximate a given element of the center
of $\Qpi(\Sigma_{g,1})$ by a polynomial in $\zeta$
(Corollary \ref{approx}).
In \S4 we prove Theorem \ref{main}.
A rough idea is as follows. Any element of
$Z(\mathbb{Q}\hat{\pi}(\Sigma_{\infty,1}))$ lies in
$Z(\mathbb{Q}\hat{\pi}(\Sigma_{g,1}))$ for some $g$.
By the result in \S3, this element is approximated by a
polynomial in $\zeta$. But we easily see the image of any positive power of $\zeta$
by the inclusion $i^g_{\infty}\colon \Sigma_{g,1}\rightarrow \Sigma_{\infty,1}$
does not lie in $Z(\mathbb{Q}\hat{\pi}(\Sigma_{\infty,1}))$,
and conclude the element must be a multiple of the constant loop.
In \S5, we remark the bracket introduced in \S3 naturally
extends to the bracket on the space of {\it linear chord diagrams}.

\vspace{0.3cm}
\noindent \textbf{Acknowledgments.}
First of all, the authors wish to express their gratitude to
Alex Bene, since some of the notions in \S3.2 are
due to discussions with him, and Yasushi Kasahara,
who pointed out a crucial error in the previous version of this paper.
The authors also would like to thank Naoko Kamada
and Shigeyuki Morita for fruitful conversations.
The first-named author is grateful to Yasto Kimura for valuable comments
about chord diagrams,
and the second-named author would like to thank Moira Chas
for helpful discussions about the Goldman Lie algebra.
Finally, the bracket introduced in Appendix was
discovered during a discussion
with Robert Penner and Jorgen Andersen at QGM, Aarhus University. 
The authors would like to thank them for helpful suggestions, 
and QGM for kind hospitality.

The first-named author is partially supported by the Grant-in-Aid for
Scientific Research (A) (No.18204002) and (A) (No.22244005) from the
Japan Society for Promotion of Sciences. The second-named
author is supported by JSPS Research Fellowships
for Young Scientists (22$\cdot$4810).

\tableofcontents

\section{Symplectic expansion and formal symplectic geometry}\label{s:sympl}

In this section we fix an integer $g \geq 1$,
and simply write $\Sigma =\Sigma_{g,1}$. 
Choose a basepoint $*$ on the boundary $\partial\Sigma$. 
The fundamental group $\pi := \pi_1(\Sigma,*)$ is a free group 
of rank 2g. The set $\hat{\pi}=\hat{\pi}(\Sigma)$ is exactly 
the set of conjugacy classes in the group $\pi$. 
We denote by $\vert\ \vert\colon \mathbb{Q}\pi \to \Qpi$ 
the natural projection. 

\subsection{Symplectic expansion}

We begin by recalling the notion of 
a symplectic expansion introduced by Massuyeau \cite{Mas}.
Let $H := H_1(\Sigma;\mathbb{Q})$ be the first 
homology group of $\Sigma$. $H$ is naturally isomorphic to 
$H_1(\pi;\mathbb{Q})\cong \pi^{\rm abel}\otimes_{\mathbb{Z}}
\mathbb{Q}$, the first homology group of $\pi$. 
Here $\pi^{\rm abel} = \pi/[\pi,\pi]$  
is the abelianization of $\pi$. 
Under this identification, we write
$$
[x] := (x \bmod [\pi,\pi]) \otimes_{\mathbb{Z}}1 \in H,
\quad\mbox{for $x \in \pi$}.
$$
Let $\T$ be the completed tensor algebra generated by $H$. 
Namely $\T = \prod^\infty_{m=0} H^{\otimes m}$, where
$H^{\otimes m}$ is the tensor space of degree $m$. 
This is a complete Hopf algebra over $\mathbb{Q}$
whose coproduct $\Delta\colon \T\to\T\hotimes\T$ is given
by $\Delta(X)=X\hotimes 1+1\hotimes X$, $X\in H$. 
Here $\T\hotimes\T$ is the completed tensor product of the two $\T$'s. 
The algebra $\T$ has a decreasing filtration given by 
$$
\T_p := \prod_{m\geq p}H^{\otimes m}, \quad 
\mbox{for $p \geq 1$}.
$$
An element $u \in \T$ is called group-like if $\Delta u 
= u\hat{\otimes} u$. As is known, the set of group-like 
elements is a subgroup of the multiplicative group 
of the algebra $\T$. 
We regard $\zeta$ as a based loop with basepoint $*$. 
If we choose a symplectic generating system 
$\{\alpha_1, \beta_1, \dots, \alpha_g, \beta_g\}$ 
of the fundamental group $\pi$ as in Figure 1, we have $\zeta = \prod^g_{i=1}
\alpha_i\beta_i{\alpha_i}^{-1}{\beta_i}^{-1}$. Here, for $\gamma_1$ and 
$\gamma_2 \in \pi$, the product $\gamma_1\gamma_2 \in \pi$ 
is defined to be the based homotopy class of the loop 
traversing first $\gamma_1$ and then $\gamma_2$. 
The intersection form on the homology group $H$ defines 
the symplectic form 
$$
\omega = \sum^g_{i=1}A_iB_i - B_iA_i \in H^{\otimes 2},
$$
where $A_i = [\alpha_i]$ and $B_i = [\beta_i] \in H$.
Here and throughout this paper we often omit $\otimes$ to express tensors.
The exponential map $\exp\colon \T_1 \to \T$ is defined by $\exp(u) =
\sum^\infty_{k=0}(1/k!)u^k \in \T$ for $u \in \T_1$.

\begin{figure}
\begin{center}
\caption{symplectic generating system}

\vspace{0.2cm}

\unitlength 0.1in
\begin{picture}( 37.5000, 18.2000)(  2.0000,-18.7000)
%
\special{pn 13}%
\special{ar 3750 1070 200 800  0.0000000 6.2831853}%
%
\special{pn 13}%
\special{ar 1000 1070 800 800  1.5707963 4.7123890}%
%
\special{pn 13}%
\special{ar 1000 1070 300 300  0.0000000 6.2831853}%
%
\special{pn 13}%
\special{ar 2400 1070 300 300  0.0000000 6.2831853}%
%
\special{pn 8}%
\special{ar 1000 1070 500 500  1.5707963 6.2831853}%
%
\special{pn 8}%
\special{ar 2400 1070 500 500  1.5707963 6.2831853}%
%
\special{pn 8}%
\special{ar 2000 1070 500 500  1.5707963 3.1415927}%
%
\special{pn 8}%
\special{ar 3400 1070 500 500  1.5707963 3.1415927}%
%
\special{pn 8}%
\special{ar 1800 1070 500 500  1.5707963 3.1415927}%
%
\special{pn 8}%
\special{ar 3200 1070 500 500  1.5707963 3.1415927}%
%
\special{pn 8}%
\special{ar 2200 1070 500 500  1.5707963 3.1415927}%
%
\special{pn 8}%
\special{ar 3600 1070 500 500  1.5707963 3.1415927}%
%
\special{pn 8}%
\special{pa 2700 1070}%
\special{pa 2700 1038}%
\special{pa 2700 1006}%
\special{pa 2700 974}%
\special{pa 2700 942}%
\special{pa 2700 910}%
\special{pa 2700 878}%
\special{pa 2700 846}%
\special{pa 2700 814}%
\special{pa 2698 782}%
\special{pa 2698 750}%
\special{pa 2698 718}%
\special{pa 2700 686}%
\special{pa 2702 654}%
\special{pa 2706 622}%
\special{pa 2710 590}%
\special{pa 2716 558}%
\special{pa 2726 526}%
\special{pa 2736 494}%
\special{pa 2746 464}%
\special{pa 2760 434}%
\special{pa 2776 406}%
\special{pa 2794 380}%
\special{pa 2812 356}%
\special{pa 2834 332}%
\special{pa 2858 310}%
\special{pa 2882 288}%
\special{pa 2900 270}%
\special{sp 0.070}%
%
\special{pn 8}%
\special{pa 3100 1070}%
\special{pa 3100 1038}%
\special{pa 3100 1006}%
\special{pa 3100 974}%
\special{pa 3100 942}%
\special{pa 3100 910}%
\special{pa 3100 878}%
\special{pa 3100 846}%
\special{pa 3102 814}%
\special{pa 3102 782}%
\special{pa 3102 750}%
\special{pa 3102 718}%
\special{pa 3102 686}%
\special{pa 3100 654}%
\special{pa 3096 622}%
\special{pa 3090 590}%
\special{pa 3084 558}%
\special{pa 3076 526}%
\special{pa 3066 494}%
\special{pa 3054 464}%
\special{pa 3040 434}%
\special{pa 3026 406}%
\special{pa 3008 380}%
\special{pa 2988 356}%
\special{pa 2966 332}%
\special{pa 2944 310}%
\special{pa 2920 288}%
\special{pa 2900 270}%
\special{sp}%
%
\special{pn 8}%
\special{pa 1300 1070}%
\special{pa 1300 1038}%
\special{pa 1300 1006}%
\special{pa 1300 974}%
\special{pa 1300 942}%
\special{pa 1300 910}%
\special{pa 1300 878}%
\special{pa 1300 846}%
\special{pa 1300 814}%
\special{pa 1298 782}%
\special{pa 1298 750}%
\special{pa 1298 718}%
\special{pa 1300 686}%
\special{pa 1302 654}%
\special{pa 1306 622}%
\special{pa 1310 590}%
\special{pa 1316 558}%
\special{pa 1326 526}%
\special{pa 1336 494}%
\special{pa 1346 464}%
\special{pa 1360 434}%
\special{pa 1376 406}%
\special{pa 1394 380}%
\special{pa 1412 356}%
\special{pa 1434 332}%
\special{pa 1458 310}%
\special{pa 1482 288}%
\special{pa 1500 270}%
\special{sp 0.070}%
%
\special{pn 8}%
\special{pa 1700 1070}%
\special{pa 1700 1038}%
\special{pa 1700 1006}%
\special{pa 1700 974}%
\special{pa 1700 942}%
\special{pa 1700 910}%
\special{pa 1700 878}%
\special{pa 1700 846}%
\special{pa 1702 814}%
\special{pa 1702 782}%
\special{pa 1702 750}%
\special{pa 1702 718}%
\special{pa 1702 686}%
\special{pa 1700 654}%
\special{pa 1696 622}%
\special{pa 1690 590}%
\special{pa 1684 558}%
\special{pa 1676 526}%
\special{pa 1666 494}%
\special{pa 1654 464}%
\special{pa 1640 434}%
\special{pa 1626 406}%
\special{pa 1608 380}%
\special{pa 1588 356}%
\special{pa 1566 332}%
\special{pa 1544 310}%
\special{pa 1520 288}%
\special{pa 1500 270}%
\special{sp}%
%
\special{pn 8}%
\special{pa 1000 1570}%
\special{pa 3600 1570}%
\special{fp}%
%
\special{pn 13}%
\special{pa 1000 1870}%
\special{pa 3746 1870}%
\special{fp}%
%
\special{pn 13}%
\special{pa 1000 270}%
\special{pa 3746 270}%
\special{fp}%
%
\special{pn 20}%
\special{sh 1}%
\special{ar 3600 1570 10 10 0  6.28318530717959E+0000}%
\special{sh 1}%
\special{ar 3600 1570 10 10 0  6.28318530717959E+0000}%
%
\special{pn 8}%
\special{pa 1676 770}%
\special{pa 1700 670}%
\special{fp}%
\special{pa 1700 670}%
\special{pa 1726 770}%
\special{fp}%
%
\special{pn 8}%
\special{pa 3076 770}%
\special{pa 3100 670}%
\special{fp}%
\special{pa 3100 670}%
\special{pa 3126 770}%
\special{fp}%
%
\special{pn 8}%
\special{pa 926 546}%
\special{pa 1026 570}%
\special{fp}%
\special{pa 1026 570}%
\special{pa 926 596}%
\special{fp}%
%
\special{pn 8}%
\special{pa 2326 546}%
\special{pa 2426 570}%
\special{fp}%
\special{pa 2426 570}%
\special{pa 2326 596}%
\special{fp}%
%
\special{pn 8}%
\special{pa 3600 1700}%
\special{pa 3660 1784}%
\special{fp}%
\special{pa 3660 1784}%
\special{pa 3646 1682}%
\special{fp}%
\put(5.6000,-6.0500){\makebox(0,0)[lb]{$\alpha_1$}}%
\put(14.4500,-2.3000){\makebox(0,0)[lb]{$\beta_1$}}%
\put(21.0000,-5.3000){\makebox(0,0)[lb]{$\alpha_2$}}%
\put(28.2500,-2.2000){\makebox(0,0)[lb]{$\beta_2$}}%
\put(36.3500,-14.7500){\makebox(0,0)[lb]{$*$}}%
\put(34.2500,-17.8500){\makebox(0,0)[lb]{$\zeta$}}%
%
\special{pn 8}%
\special{ar 3350 460 100 100  6.2831853 6.2831853}%
\special{ar 3350 460 100 100  0.0000000 4.7123890}%
%
\special{pn 8}%
\special{pa 3450 460}%
\special{pa 3390 490}%
\special{fp}%
\special{pa 3450 460}%
\special{pa 3486 516}%
\special{fp}%
\end{picture}%
\end{center}
\end{figure}
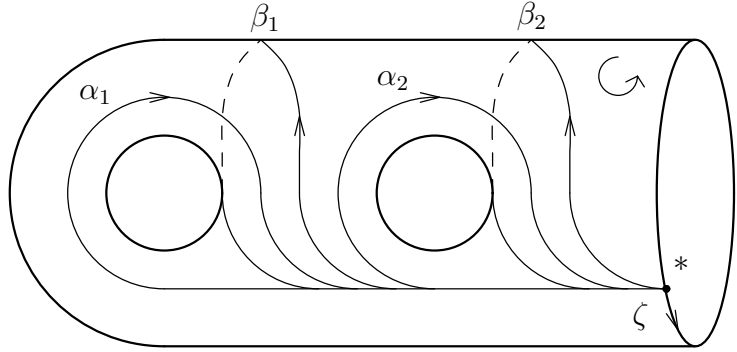

\begin{dfn}\label{sympl-exp}{\rm (Massuyeau \cite{Mas})} A symplectic
expansion
$\theta$ of the fundamental group $\pi$ of the surface $\Sigma$ is a map
$\theta\colon \pi\to \T$ satisfying the conditions
\begin{enumerate}
\item[{\rm (1)}] $\theta(x) \equiv 1 + [x] \pmod{\T_2}$ for any $x \in \pi$,
\item[{\rm (2)}] $\theta(xy) = \theta(x)\theta(y)$ for any $x$ and $y \in \pi$,
\item[{\rm (3)}] $\theta(x)$ is group-like for any $x \in \pi$,
\item[{\rm (4)}] $\theta(\zeta) = \exp(\omega)$.
\end{enumerate}
\end{dfn}
Symplectic expansions do exist \cite{Mas} Lemma 2.16, 
and they are infinitely many \cite{KK} Proposition 2.8.1.
Several constructions of a symplectic expansion
are known; {\it harmonic Magnus expansions} \cite{Ka2}
via a transcendental method, a construction in \cite{Mas}
via the {\it LMO functor}; also there is an elementary method to
associate a symplectic expansion with
any (not necessary symplectic) free generators
of $\pi$ \rm \cite{Ku}.

A map $\theta\colon \pi \to \T$ satisfying the conditions (1) and (2) is called 
a {\it Magnus expansion} of the free group $\pi$ \cite{Ka}. 

\subsection{Formal symplectic geometry}

We recall the ``associative" formal symplectic geometry 
$\ag$ introduced by Kontsevich \cite{Kon}. 
Let $N\colon \T \to \T_1$ be a linear map defined by 
$$
N\vert_{H^{\otimes n}} = \sum^{n-1}_{k=0}\nu^k, \quad
\mbox{$n \geq 1$},
$$
where $\nu$ is the cyclic permutation given by 
$X_1X_2\cdots X_n \mapsto X_2\cdots X_nX_1$ for $X_i \in H$, $n \geq 1$, 
and $N\vert_{H^{\otimes 0}}= 0$. 
By definition, a derivation on $\T$ is a linear map 
$D\colon \T\to\T$ 
satisfying the Leibniz rule:
$$
D(u_1u_2) = D(u_1)u_2 + u_1D(u_2),
$$
for $u_1, u_2 \in \T$.
Since $\T$ is freely generated by $H$ as a complete algebra, 
any derivation on $\T$ is uniquely determined by its values on $H$, 
and the space of derivations of $\T$ is identified with 
$\Hom(H, \T)$. By the Poincar\'e duality, 
$\T_1 \cong H\otimes\T$ is identified with $\Hom(H, \T)$: 
\begin{equation}\label{T1-Hom}
\T_1 \cong H\otimes\T\overset\cong\to \Hom(H,\T), \quad
X \otimes u \mapsto (Y \mapsto (Y \cdot X)u).
\end{equation}
Here $(\ \cdot\ )$ is the intersection pairing on $H = H_1(\Sigma;
\mathbb{Q})$.\par 
Let $\agminus = \deromega(\T)$ be the space of derivations 
on $\T$ killing the symplectic form $\omega$. 
In view of (\ref{T1-Hom}) any derivation D is written as 
$$
D = \sum^g_{i=1}B_i\otimes D(A_i) - A_i \otimes D(B_i). 
$$
Since $-D(\omega) = \sum^g_{i=1}[B_i,D(A_i)] - [A_i,D(B_i)]$, 
we have $\agminus = \Ker([\ ,\ ]\colon H \otimes\T\to \hat{T})$. 
It is easy to see $\Ker([\ ,\ ]) = N(\T_1)$ (see \cite{KK} 
Lemma 2.6.2 (4)).
Hence we can write
\begin{equation}
\agminus = \Ker([\ ,\ ]\colon H\otimes\T \to \T) = N(\T_1).
\label{ag-}
\end{equation}
The Lie subalgebra $\ag := N(\T_2)$ is nothing but (the completion of)
what Kontsevich \cite{Kon} calls $a_g$. 
By a straightforward computation, the bracket on $\agminus$ as derivations
is given as follows.
\begin{lem}\label{[Nu,Nv]} We have 
\begin{eqnarray*}
&&[N(X_1\cdots X_n), N(Y_1\cdots Y_m)] \\
&=& -\sum^n_{s=1}\sum^m_{t=1}(X_s\cdot Y_t)N(X_{s+1}\cdots X_nX_1\cdots
X_{s-1}Y_{t+1}\cdots Y_mY_1\cdots Y_{t-1})
\end{eqnarray*}
for $X_1,\ldots,X_n,Y_1,\ldots,Y_m \in H$. 
\end{lem}
We introduce a bilinear map
$\mathcal{B}\colon H^{\otimes n}\times H^{\otimes m}
\to N(H^{\otimes (n+m-2)})$
by 
\begin{equation}
\mathcal{B}(X_1\cdots X_n, Y_1\cdots Y_m)
:= -(X_1\cdot Y_1)N(X_2\cdots X_nY_{2}\cdots Y_m)
\label{calB}
\end{equation}
for $X_s, Y_t \in H$. Then, Lemma \ref{[Nu,Nv]} is written as
\begin{equation}
[Nu, Nv] = \sum^{n-1}_{s=0}\sum^{m-1}_{t=0}\mathcal{B}(\nu^su, \nu^tv) 
= \mathcal{B}(Nu, Nv)
\label{[Nu,Nv]B}
\end{equation}
for $u \in H^{\otimes n}$ and $v \in H^{\otimes m}$.
\par

\subsection{``Completion'' of the Goldman Lie algebra}
The following result is proved in \cite{KK}.
\begin{thm}[\cite{KK} Theorem 1.2.1]
\label{completion}
For any symplectic expansion $\theta$, 
the map
$$
-N\theta\colon \mathbb{Q}\hat{\pi}\to N(\T_1) = \agminus, \quad
\pi \ni x\mapsto -N\theta(x) \in N(\T_1)
$$
is a well-defined Lie algebra homomorphism. The kernel is the subspace
$\mathbb{Q}1$ spanned by the constant loop $1$, and the image is dense in
$N(\T_1) =\agminus$ with respect to the $\T_1$-adic topology. 
\end{thm}
By this theorem, we may regard the formal symplectic geometry 
$\agminus$ as a certain kind of completion of the Goldman Lie algebra
$\Qpi$.  We introduce a decreasing filtration of the Goldman Lie algebra 
$\Qpi$ defined by 
$$
\Qpi(p) := (N\theta)^{-1}N(\T_p), \quad 
\mbox{for $p \geq 1$}.
$$
Since $N(\T_p)$ is a Lie subalgebra of $\agminus =
N(\T_1)$,  the subspace $\Qpi(p)$ is also a Lie subalgebra of 
$\Qpi$.   Let $\theta'\colon \pi \to \T$ be another Magnus expansion 
which is not necessarily symplectic.
We denote by $[\T,\T]$ the derived ideal of $\T$ as a Lie algebra, 
in other words, $[\T,\T]$ is the vector subspace generated by the set
$\{uv-vu;\, u,v\in\T\}$. Let $\varepsilon \colon \T \to H^0=\mathbb{Q}$
be the augmentation.
\begin{lem}\label{characterization} Fix $p \geq 1$. For $u \in
\mathbb{Q}\pi$, the followings are equivalent.
\begin{enumerate}
\item[{\rm (1)}] $\vert u\vert \in \Qpi(p)$, namely, $N\theta(u) \in N(\T_p)$. 
\item[{\rm (2)}] $\theta(u) -\varepsilon(\theta(u)) \in \T_p+[\T,\T]$.
\item[{\rm (3)}] $\theta'(u) -\varepsilon(\theta'(u)) \in \T_p+[\T,\T]$.
\end{enumerate}
In particular, the filtration $\{\Qpi(p)\}^\infty_{p=1}$ 
is independent of the choice of a Magnus expansion.
\end{lem}
\begin{proof}
We have $N(X_1\cdots X_n)-nX_1\cdots X_n = 
\sum^n_{i=1}(X_i\cdots X_nX_1\cdots X_{i-1}- X_1\cdots X_n) = 
\sum^n_{i=2}[X_i\cdots X_n,X_1\cdots X_{i-1}] \in [\T,\T]$ 
for $X_i\in H$. This means $Nu -nu \in [\T,\T]$ for any $u \in
H^{\otimes n}$.
If $u \in \Ker N\cap H^{\otimes n}$, then $u =-\frac{1}{n}(Nu -nu)
\in  [\T,\T]$. Clearly $N[\T,\T] = 0$. Hence we have 
\begin{equation}
0\to [\T,\T] \to \T_1\overset{N}\to \T_1\quad\mbox{(exact)}.
\label{N[]}
\end{equation}
In particular, we have $(N\vert_{\T_1})^{-1}(N(\T_p)) = \T_p+[\T,\T]$, 
which implies the conditions (1) and (2) are equivalent.\par
As was proved in \cite{Ka} Theorem 1.3, 
there exists a filter-preserving algebra automorphism $U$ of $\T$ 
satisfying the equation $\theta' = U\circ\theta$. 
Then we have $U(\T_p+[\T,\T]) = \T_p+[\T,\T]$. 
Hence the conditions (2) and (3) are equivalent. 
This proves the lemma.
\end{proof}

Let $I\pi$ be the augmentation ideal of the group ring 
$\mathbb{Q}\pi$, i.e., the kernel of the augmentation
$\varepsilon\colon \mathbb{Q}\pi \to \mathbb{Q}$. 
It is easy to show $\Qpi(p) = \vert \mathbb{Q}1+(I\pi)^p\vert$, 
from which it also follows $\Qpi(p)$ is independent of 
a Magnus expansion. 
As a corollary of Theorem \ref{completion}, we have
\begin{eqnarray}
&& \bigcap^\infty_{p=1}\Qpi(p) = \Ker N\theta = \mathbb{Q}1,
\quad\mbox{and}
\label{intersection}\\
&& Z(\Qpi) \subset (N\theta)^{-1}Z(\agminus).
\label{incl-Z}
\end{eqnarray}
In view of this corollary (\ref{incl-Z}), we are led to consider
the center $Z(\agminus)$ of the Lie algebra $\agminus$. 
The subspace $N(H^{\otimes 2})$ of $\agminus$ is a Lie subalgebra 
naturally isomorphic to the Lie algebra of the symplectic group, 
$\fsp := \mathfrak{sp}_{2g}(\mathbb{Q})$. 
Hence $Z(\agminus)$ is included in the $\mathfrak{sp}$-invariants
$(\agminus)^{\fsp} = (\ag)^{\fsp}$, i.e., the tensors anihilated by
the action of $\mathfrak{sp}$.
Here we use the fact $H^{\fsp} = 0$. The subspace $(\ag)^{\fsp}$ is 
a Lie subalgebra of $\ag$. Thus we obtain
\begin{equation}
Z(\agminus) \subset Z((\ag)^{\fsp}). 
\label{incl-ag}
\end{equation}

\section{Lie algebra of oriented chord diagrams}\label{s:chord}

In this section, we describe the Lie algebra $(\ag)^{\fsp}$ 
in a stable range by introducing the Lie algebra of oriented chord
diagrams. 
Following Morita \cite{KM} \cite{MoS} \cite{MoG} and Kontsevich
\cite{Kon},  we make the symplectic form $\omega$ correspond to a
labeled chord. 

\subsection{The $\fsp$-invariant tensors}

Under the identification $\ag = N(\T_2)$, we denote 
$(\ag)_{(n)} := \ag \cap H^{\otimes n} = N(H^{\otimes n}) \subset 
H^{\otimes n}$ for $n \geq 2$. 
We begin by recalling the $\fsp$-invariant tensors in the space 
$H^{\otimes n}$. 
It is a classical result of Weyl \cite{W} ch.\ VI, \S1, that
the space of $\fsp$-invariant tensors in $H^{\otimes n}$ is 
zero if $n$ is odd, and generated by {\it linear chord 
diagrams} of $n/2$ chords if $n$ is even. 
Let $m$ be a positive integer. 
A {\it linear chord diagram} of $m$ chords is a decomposition 
of the set of vertices $\{1, 2, \dots, 2m\}$ into $m$
{\it unordered} pairs 
$\{\{i_1, j_1\}, \{i_2, j_2\}, \dots, \{i_m, j_m\}\}$. 
Further, a {\it labeled linear chord diagram} 
of $m$ chords $C$ is a set of $m$ {\it ordered} pairs 
$\{(i_1, j_1), (i_2, j_2), \dots, (i_m, j_m)\}$
satisfying the condition 
$\{i_1,\dots, i_m, j_1,\dots, j_m\} = \{1, 2, \dots, 2m\}$. 
We denote by $\overline{C}$ the underlying linear chord diagram of $C$,
$\overline{C} := \{\{i_1, j_1\}, \{i_2, j_2\}, \dots, \{i_m, j_m\}\}$. 
An $\fsp$-invariant tensor $a(C) \in H^{\otimes 2m}$ is defined by 
$$
a(C) := \begin{pmatrix}1&2&\cdots&2m-1&2m\\
i_1&j_1&\cdots&i_m&j_m
\end{pmatrix}(\omega^{\otimes m}) \in (H^{\otimes 2m})^{\fsp}.
$$
Let $C'$ be a labeled linear chord diagram obtained from $C$ by a 
single label change. Namely, we have 
$$
C' = \{(i_1, j_1),\dots, (i_{k-1}, j_{k-1}), (j_k, i_k), (i_{k+1},
j_{k+1}),\dots, (i_m, j_m)\}
$$
for a single $k$. Clearly we have $\overline{C'} = \overline{C}$ and $a(C') = -a(C)$. 
We denote by $\mathcal{LC}_m$ the $\mathbb{Q}$-linear space 
spanned by the labeled linear chord diagrams of $m$ chords 
modulo the linear subspace generated by the set 
$$
\{C+C'; \mbox{$C'$ is obtained from $C$ by a single
label change.}\},
$$
and call it {\it the space of oriented linear chord 
diagrams of $m$ chords}. We have a natural map
$$
a\colon \mathcal{LC}_m \to (H^{\otimes 2m})^{\fsp}, \quad
C \mapsto a(C).
$$
Now we have 
\begin{lem}\label{Weyl-Morita}
The map $a\colon \mathcal{LC}_m \to (H^{\otimes 2m})^{\fsp}$ is 
\begin{enumerate}
\item[{\rm (1)}] surjective for any $m \geq 1$, and
\item[{\rm (2)}] an isomorphism if and only if $m \leq g$.
\end{enumerate}
\end{lem}
The assertion (1) and the ``if'' part of (2) are Weyl's result stated
above, while the ``only if'' part of (2) is due to Morita 
\cite{MoG} p.797,  Proposition 4.1. See also \cite{MoS}
p.361, Lemma 4.1. \par
Let $\nu \in \mathfrak{S}_{2m}$ be the cyclic permutation 
introduced in \S2.2
$$
\nu = \begin{pmatrix}1&2&3&\cdots&2m\\
2m&1&2&\cdots&2m-1
\end{pmatrix}.
$$
We denote by $Z_{2m}$ the cyclic subgroup generated by $\nu$ 
in the group $\mathfrak{S}_{2m}$. 
For a labeled linear chord diagram $C = \{(i_1, j_1), (i_2, j_2), 
\dots,(i_m, j_m)\}$, we define 
$$
\nu^s(C) := \{(\nu^s(i_1), \nu^s(j_1)), (\nu^s(i_2), \nu^s(j_2)), \dots,
(\nu^s(i_m), \nu^s(j_m))\}, \quad s \in \mathbb{Z}.
$$
Clearly we have $a(\nu^sC) = \nu^sa(C)$. 
This action descends to an action of $Z_{2m}$ on the space 
$\mathcal{LC}_m$. We define the space $\mathcal{C}_m$ as the 
$Z_{2m}$-invariants in $\mathcal{LC}_m$
$$
\mathcal{C}_m := (\mathcal{LC}_m)^{Z_{2m}},
$$
and call it {\it the space of oriented chord diagrams of $m$ chords}.
If $m=1$ and $C = \{(1,2)\}$, then 
$\nu(C) = -C \in \mathcal{LC}_1$. Hence we have $\mathcal{C}_1 = 0$. 
We define 
\begin{equation*}
\mathcal{C} := \prod^\infty_{m=2}\mathcal{C}_m.
\end{equation*}
A {\it labeled chord diagram} of $m$ chords $\mathcal{N}(C)$
is a collection of $2m$ labeled linear chord diagrams
$C,\nu(C),\ldots,\nu^{2m-1}(C)$ for some $C$ with $m$ chords
(to be more precise, we consider $\mathcal{N}(C)$ as an element of
the $2m$-th symmetric product of the set of labeled linear chord diagrams
of $m$ chords).
We also denote $N(C) := \sum^{2m-1}_{s=0}\nu^s(C) \in \mathcal{C}_m$. 
We have $\mathcal{N}(\nu(C))=\mathcal{N}(C)$ and
$$
a(N(C)) = N(a(C)) \in (H^{\otimes 2m})^{Z_{2m}} = N(H^{\otimes 2m})
= (\ag)_{(2m)}. 
$$

\begin{dfn}
\label{index}
For a labeled linear chord diagram $C$ of $m$ chords, define the index of $C$
as the cardinality of the set
$\{\overline{\nu^s(C)}; 0 \leq s \leq 2m-1\}$.
We also define the index of $\mathcal{N}(C)$ as the
index of one of the diagrams in $\mathcal{N}(C)$.
We say a diagram is of maximal index if its index is
twice the number of chords.
\end{dfn}

In general, the index of a chord diagram divides twice the number of chords.
Clearly the index of $\mathcal{N}(C)$ is independent of the choice
of a diagram.

\begin{lem}\label{odd}
Let $C$ be a labeled linear chord diagram of $m$ chords. Then $N(C)=0 
\in \mathcal{C}_m$ if and only if $C$ is of odd index. 
\end{lem}
\begin{proof}
We denote by $\overline{C}^\flat$ the labeled linear chord diagram on the
underlying linear chord diagram $\overline{C}$ with {\it the standard label}, 
which means $i_k < j_k$ for any $k$.  
Clearly we have $N(C) = \pm N(\overline{C}^\flat)$. 
Let $l$ be the index of $C$. Then, 
since $\nu(\overline{C}^\flat) = -\overline{\nu(C)}^\flat 
\in \mathcal{C}_m$, we have
\begin{eqnarray*}
&& N(\overline{C}^\flat) = \sum^{2m-1}_{s=0}\nu^s(\overline{C}^\flat)
= \sum^{2m-1}_{s=0}(-1)^s\overline{\nu^s(C)}^\flat
= \sum^{(2m/l)-1}_{i=0}(-1)^{li}
\sum^{l-1}_{j=0}(-1)^j\overline{\nu^j(C)}^\flat\\
&=& \left\{
\begin{array}{ll}
\frac{\displaystyle 2m}{\displaystyle
l}\sum^{l-1}_{j=0}(-1)^j\overline{\nu^j(C)}^\flat&
\quad
\mbox{if $l$ is even,}\\ 0,& \quad \mbox{if $l$ is odd.}\\
\end{array}\right.
\end{eqnarray*}
Here we remark $\overline{\nu^j(C)}^\flat$, $0 \leq j \leq l-1$, are
linearly  independent. This proves the lemma.
\end{proof}
The following is a corollary of Lemma \ref{Weyl-Morita}.
\begin{lem}\label{a-isom} The map 
$a\colon \mathcal{C}_m \to (\ag)_{(2m)}^\fsp$, $N(C) \mapsto a(N(C))$, is 
\begin{enumerate}
\item[{\rm (1)}] surjective for any $m \geq 1$, and
\item[{\rm (2)}] an isomorphism if $m \leq g$.
\end{enumerate}
\end{lem}
\begin{proof}
The map $a\colon \mathcal{C}_m \to (\ag)_{(2m)}^\fsp$ is
the restriction of $a\colon \mathcal{LC}_m \to (H^{\otimes 2m})^{\fsp}$,
and the map $N\colon (H^{\otimes 2m})^{\fsp} \to (\mathfrak{a}_g)_{(2m)}^{\fsp}$
is surjective since the surjection $N\colon H^{\otimes 2m}\to (\mathfrak{a}_g)_{(2m)}$
is $\fsp$-equivariant. Hence the assertions follow from Lemma \ref{Weyl-Morita}.
\end{proof}

Hence the map $a\colon \mathcal{C} \to (\ag)^{\fsp}$ is an isomorphism 
in a stable range. So we compute the bracket on the Lie algebra 
$(\ag)^{\fsp}$ by means of the stable isomorphism $a$. 
Let $C$ and $C'$ be labeled chord diagrams given by 
$$
C = \{(i_1, j_1), (i_2, j_2), \dots, (i_m, j_m)\} \quad\mbox{and}\quad
C' = \{(a_1, b_1), (a_2, b_2), \dots, (a_l, b_l)\}.
$$
Then, by the formula (\ref{[Nu,Nv]B}), we have 
$$
[a(N(C)), a(N(C'))] = [N(a(C)), N(a(C'))] 
= \sum^{2m-1}_{s=0}\sum^{2l-1}_{t=0}\mathcal{B}(a(\nu^sC), 
a(\nu^tC')). 
$$
In order to describe $\mathcal{B}(a(C), a(C'))$, we define 
an amalgamation 
of two labeled linear chord diagrams as follows. 
We may assume $1 \in \{i_1, j_1\}\cap \{a_1, b_1\}$
without loss of generality.   
A labeled chord diagram of $m+l-1$ chords $C\ast C'$ is defined by 
$$
\{(x, y), (i_2-1, j_2-1), \dots, (i_m-1, j_m-1), 
(a_2+2m-2, b_2+2m-2), \dots, (a_l+2m-2, b_l+2m-2)\},
$$
where 
$$
(x, y) := \left\{
\begin{array}{ll}
(b_1+2m-2, j_1-1), &\quad\mbox{if $i_1=a_1=1$,}\\
(j_1-1, a_1+2m-2), &\quad\mbox{if $i_1=b_1=1$,}\\
(i_1-1, b_1+2m-2), &\quad\mbox{if $j_1=a_1=1$,}\\
(a_1+2m-2, i_1-1), &\quad\mbox{if $j_1=b_1=1$.}\\
\end{array}\right.
$$
We call it the {\it amalgamation} 
of the labeled linear chord diagrams $C$ and $C'$.  
Then we have 
$$
\mathcal{B}(a(C), a(C')) = Na(C\ast C').
$$
In fact, if we define a bilinear map 
$\mathcal{B}'\colon H^{\otimes 2}\times H^{\otimes 2}\to H^{\otimes 2}$
by $\mathcal{B}'(X_1X_2, Y_1Y_2) := -(X_1\cdot Y_1)X_2Y_2$, $X_i, Y_j 
\in H$, then we have $\mathcal{B}'(\omega, \omega) = -\omega$. 
This means $(x, y)$ should be $(b_1+2m-2, j_1-1)$ in the case 
$i_1=a_1=1$. Similar observations hold for the other three cases. 
Hence we obtain
\begin{lem}\label{a-bracket}
$$
[a(N(C)), a(N(C'))] 
= \sum^{2m-1}_{s=0}\sum^{2l-1}_{t=0}a(N((\nu^sC)\ast(\nu^tC'))).
$$
\end{lem}
Here it should be remarked that the right hand side in the above equality does {\it not} 
depend on the genus $g$. Since the map $a$ is a stable isomorphism,
the whole of the maps $a$ 
induces a Lie algebra structure on the space $\mathcal{C}$. 
The bracket is given by 
\begin{equation}
[N(C), N(C')] =
\sum^{2m-1}_{s=0}\sum^{2l-1}_{t=0}N((\nu^sC)\ast(\nu^tC')).
\label{N(C)-bracket}
\end{equation}
From Lemma \ref{a-bracket} the map $a: \mathcal{C} \to {\ag}^{\fsp}$ is a
Lie algebra homomorphism for each $g \geq 1$. In the next subsection, we
will  give a diagrammatic description of 
the Lie algebra $\mathcal{C}$, 
which will enable us to compute the center $Z((\ag)^{\fsp})$ 
in a stable range. 

\subsection{The center of the Lie algebra of oriented chord diagrams}
In this subsection we give a diagrammatic description of the Lie algebra structure
on $\mathcal{C}=\prod_m \mathcal{C}_m$ introduced by the formula (\ref{N(C)-bracket})
and compute its center.

We first recall the description of labeled linear chord diagrams by picture.
Let $C=\{ (i_1,j_1),(i_2,j_2),\ldots,(i_m,j_m)\}$ be a
labeled linear chord diagram of $m$ chords. Fix a closed interval
on the $x$-axis in the $xy$-plane and call it {\it the core} of the diagram.
Put $2m$ distinct points on the interior of the core, and for each
$1\le k \le m$, draw an oriented simple path, called a {\it labeled chord},
in the upper half plane from the $i_k$-th point
(with respect to the $x$-coordinate) to the $j_k$-th point.
Hereafter we identify a labeled linear chord diagram with its picture.
For example, the picture of $C=\{ (1,2),(3,5),(4,6)\}$ is as in
Figure 2.

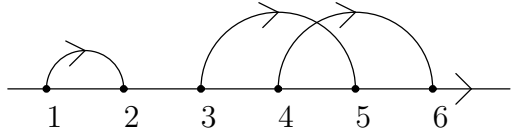
\begin{figure}
\begin{center}
\caption{labeled linear chord diagram}

\vspace{0.2cm}

\unitlength 0.1in
\begin{picture}( 26.0000,  5.3000)(  4.0000,-14.8000)
%
\special{pn 20}%
\special{sh 1}%
\special{ar 600 1400 10 10 0  6.28318530717959E+0000}%
\special{sh 1}%
\special{ar 1000 1400 10 10 0  6.28318530717959E+0000}%
\special{sh 1}%
\special{ar 1400 1400 10 10 0  6.28318530717959E+0000}%
\special{sh 1}%
\special{ar 1800 1400 10 10 0  6.28318530717959E+0000}%
\special{sh 1}%
\special{ar 2200 1400 10 10 0  6.28318530717959E+0000}%
\special{sh 1}%
\special{ar 2600 1400 10 10 0  6.28318530717959E+0000}%
%
\special{pn 8}%
\special{ar 800 1400 200 200  3.1415927 6.2831853}%
%
\special{pn 8}%
\special{ar 1800 1400 400 400  3.1415927 6.2831853}%
%
\special{pn 8}%
\special{ar 2200 1400 400 400  3.1415927 6.2831853}%
%
\special{pn 8}%
\special{pa 2200 1000}%
\special{pa 2100 950}%
\special{fp}%
%
\special{pn 8}%
\special{pa 1800 1000}%
\special{pa 1700 950}%
\special{fp}%
%
\special{pn 8}%
\special{pa 800 1200}%
\special{pa 700 1150}%
\special{fp}%
\put(6.0000,-16.0000){\makebox(0,0)[lb]{$1$}}%
\put(10.0000,-16.0000){\makebox(0,0)[lb]{$2$}}%
\put(14.0000,-16.0000){\makebox(0,0)[lb]{$3$}}%
\put(18.0000,-16.0000){\makebox(0,0)[lb]{$4$}}%
\put(22.0000,-16.0000){\makebox(0,0)[lb]{$5$}}%
\put(26.0000,-16.0000){\makebox(0,0)[lb]{$6$}}%
%
\special{pn 8}%
\special{pa 2200 1000}%
\special{pa 2120 1080}%
\special{fp}%
%
\special{pn 8}%
\special{pa 1800 1000}%
\special{pa 1720 1080}%
\special{fp}%
%
\special{pn 8}%
\special{pa 800 1200}%
\special{pa 720 1280}%
\special{fp}%
%
\special{pn 8}%
\special{pa 400 1400}%
\special{pa 3000 1400}%
\special{fp}%
%
\special{pn 8}%
\special{pa 2800 1400}%
\special{pa 2720 1320}%
\special{fp}%
%
\special{pn 8}%
\special{pa 2800 1400}%
\special{pa 2720 1480}%
\special{fp}%
\end{picture}%
\end{center}
\end{figure}

We next recall a diagrammatic description of the labeled chord diagrams
introduced in \S3.1. In this subsection,
a {\it labeled chord diagram} of $m$ chords is a diagram in the $xy$-plane
consisting of a circle, called {\it the core}, $2m$ vertices on the core,
and $m$ oriented simple paths, called {\it labeled chords}, connecting
two vertices in the disk which bounds the core, such that the ends
of the labeled chords exhaust the $2m$ vertices. We give the orientation
to the core coming from that of the disk.

Given a labeled linear chord diagram $C$, we can produce
a labeled chord diagram by connecting the two ends of the core
by a simple path in the upper half plane avoiding the $m$ chords.
We call this operation {\it the closing of $C$}. For example, the closing
of the diagram in Figure 2 is as in Figure 3.

\begin{figure}
\begin{center}
\caption{the closing}

\vspace{0.2cm}

\unitlength 0.1in
\begin{picture}( 16.6200, 16.0600)(  4.0000,-20.0600)
%
\special{pn 8}%
\special{ar 1200 1204 800 800  0.0000000 6.2831853}%
%
\special{pn 20}%
\special{sh 1}%
\special{ar 1890 804 10 10 0  6.28318530717959E+0000}%
%
\special{pn 20}%
\special{sh 1}%
\special{ar 1890 1604 10 10 0  6.28318530717959E+0000}%
%
\special{pn 20}%
\special{sh 1}%
\special{ar 520 1604 10 10 0  6.28318530717959E+0000}%
%
\special{pn 20}%
\special{sh 1}%
\special{ar 520 804 10 10 0  6.28318530717959E+0000}%
%
\special{pn 20}%
\special{sh 1}%
\special{ar 1200 404 10 10 0  6.28318530717959E+0000}%
%
\special{pn 20}%
\special{sh 1}%
\special{ar 1200 2004 10 10 0  6.28318530717959E+0000}%
%
\special{pn 8}%
\special{ar 1546 1802 400 400  2.6131589 5.7556517}%
%
\special{pn 8}%
\special{pa 1200 404}%
\special{pa 520 1604}%
\special{fp}%
%
\special{pn 8}%
\special{pa 520 804}%
\special{pa 1880 804}%
\special{fp}%
%
\special{pn 8}%
\special{pa 1400 804}%
\special{pa 1480 724}%
\special{fp}%
%
\special{pn 8}%
\special{pa 1400 804}%
\special{pa 1480 884}%
\special{fp}%
%
\special{pn 8}%
\special{pa 832 1042}%
\special{pa 800 934}%
\special{fp}%
%
\special{pn 8}%
\special{pa 832 1042}%
\special{pa 940 1012}%
\special{fp}%
%
\special{pn 8}%
\special{pa 1360 1460}%
\special{pa 1250 1452}%
\special{fp}%
\special{pa 1360 1460}%
\special{pa 1316 1564}%
\special{fp}%
%
\special{pn 8}%
\special{pa 2000 1140}%
\special{pa 2062 1234}%
\special{fp}%
\special{pa 2000 1140}%
\special{pa 1932 1230}%
\special{fp}%
\put(20.1000,-16.6000){\makebox(0,0)[lb]{$p$}}%
\end{picture}%
\end{center}
\end{figure}
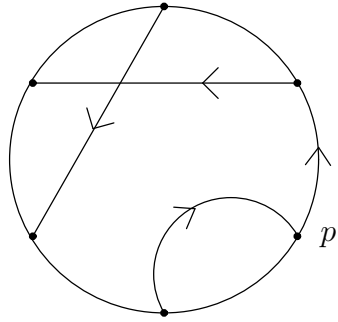

Conversely, given a vertex $p$ of a labeled chord diagram $D$,
we can produce a labeled linear chord diagram by cutting the core at
a little short of $p$ and embed the result into the $xy$-plane
so that the cut core is included in the $x$-axis and the labeled
chords are included in the upper half plane. We call
this operation {\it the cut of $D$ at $p$}, and denote the result
by $C(D,p)$. For example,
the cut of the diagram at $p$ in Figure 3 is as in Figure 4.

\begin{figure}
\begin{center}
\caption{the cut $C(D,p)$}

\vspace{0.2cm}

\unitlength 0.1in
\begin{picture}( 26.0000, 11.4600)(  4.0000,-14.8000)
%
\special{pn 20}%
\special{sh 1}%
\special{ar 600 1400 10 10 0  6.28318530717959E+0000}%
\special{sh 1}%
\special{ar 1000 1400 10 10 0  6.28318530717959E+0000}%
\special{sh 1}%
\special{ar 1400 1400 10 10 0  6.28318530717959E+0000}%
\special{sh 1}%
\special{ar 1800 1400 10 10 0  6.28318530717959E+0000}%
\special{sh 1}%
\special{ar 2200 1400 10 10 0  6.28318530717959E+0000}%
\special{sh 1}%
\special{ar 2600 1400 10 10 0  6.28318530717959E+0000}%
%
\special{pn 8}%
\special{pa 1800 1000}%
\special{pa 1700 950}%
\special{fp}%
\put(6.0000,-16.0000){\makebox(0,0)[lb]{$1$}}%
\put(10.0000,-16.0000){\makebox(0,0)[lb]{$2$}}%
\put(14.0000,-16.0000){\makebox(0,0)[lb]{$3$}}%
\put(18.0000,-16.0000){\makebox(0,0)[lb]{$4$}}%
\put(22.0000,-16.0000){\makebox(0,0)[lb]{$5$}}%
\put(26.0000,-16.0000){\makebox(0,0)[lb]{$6$}}%
%
\special{pn 8}%
\special{pa 1800 1000}%
\special{pa 1720 1080}%
\special{fp}%
%
\special{pn 8}%
\special{pa 400 1400}%
\special{pa 3000 1400}%
\special{fp}%
%
\special{pn 8}%
\special{pa 2800 1400}%
\special{pa 2720 1320}%
\special{fp}%
%
\special{pn 8}%
\special{pa 2800 1400}%
\special{pa 2720 1480}%
\special{fp}%
%
\special{pn 8}%
\special{ar 1600 1400 1000 1000  3.1415927 6.2831853}%
%
\special{pn 8}%
\special{ar 1400 1400 400 400  3.1415927 6.2831853}%
%
\special{pn 8}%
\special{ar 1800 1400 400 400  3.1415927 6.2831853}%
%
\special{pn 8}%
\special{pa 1400 1000}%
\special{pa 1300 950}%
\special{fp}%
%
\special{pn 8}%
\special{pa 1400 1000}%
\special{pa 1320 1080}%
\special{fp}%
%
\special{pn 8}%
\special{pa 1586 402}%
\special{pa 1676 334}%
\special{fp}%
%
\special{pn 8}%
\special{pa 1586 402}%
\special{pa 1680 466}%
\special{fp}%
\end{picture}%
\end{center}
\end{figure}
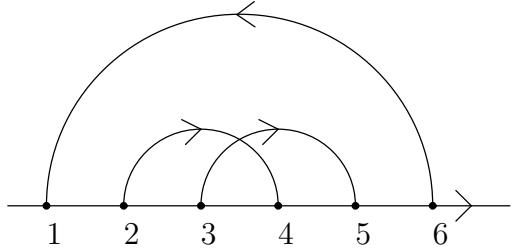

Let $D$ be a labeled chord diagram of $m$ chords, and $p_0$ a vertex of $D$.
The collection of the cuts $C(D,p)$, where $p$ runs through all the vertices
of $D$, can be written as $\nu^kC(D,p_0)$, $0\le k\le 2m-1$.
This implies the two notions of labeled chord diagrams given in \S3.1 and
here are essentially the same.
The sum $\sum_p C(D,p)\in \mathcal{LC}_m$ equals $N(C(D,p_0))$
hence is in $\mathcal{C}_m$.
Let $D^{\prime}$ be a labeled chord diagram obtained from $D$
by a single label change. Namely, $D^{\prime}$ is
obtained from $D$ by reversing the orientation
of a single labeled chord. Then $\sum_p C(D,p)=-\sum_p C(D^{\prime},p)$.
Therefore the space $\mathcal{C}_m$ is also described as
the $\mathbb{Q}$-linear space spanned by the labeled chord diagrams
of $m$ chords modulo the subspace generated by the set
$$\{D+D^{\prime};\ D^{\prime} {\rm \ is\ obtained\ from\ } D
{\rm \ by\ a\ single\ label\ change\ } \}.$$

We shall often regard a labeled chord diagram as an element of
$\mathcal{C}:=\prod_m \mathcal{C}_m$, if there is no confusion.

Let $D$ and $D^{\prime}$ be labeled chord diagrams and let $p$ and $q$ be vertices
of $D$ and $D^{\prime}$, respectively.
We shall produce a new labeled chord diagram $\mathcal{D}(D,p,D^{\prime},q)$,
which corresponds to an amalgamation in \S3.1,
by the following way.
Let $p_-$ and $p_+$ be the vertices of $D$ adjacent to $p$,
such that they are arranged as $p_- <p <p_+$ with respect to
the cyclic ordering of vertices coming from the orientation of the core.
Similarly, define $q_-$ and $q_+$. Also, let $\overline{p}$
(resp.\ $\overline{q}$) be the vertex of $D$ (resp.\ $D^{\prime}$)
which is the other end of the edge through $p$ (resp.\ $q$).

The first step is to place the cut $C(D^{\prime},q)$ on the right of
the cut $C(D,p_+)$, and regard the entirety as a labeled linear
chord diagram. The second step is to remove the two chords through $p$ or $q$,
and add a labeled chord which connects $\overline{p}$ and $\overline{q}$ instead.
The label of the added chord is determined by the rule indicated in Figure 5.

\begin{figure}
\begin{center}
\caption{the label of the added chord}

\vspace{0.2cm}

\unitlength 0.1in
\begin{picture}( 34.0300, 12.1000)(  3.9700,-16.4000)
%
\special{pn 8}%
\special{pa 400 1000}%
\special{pa 1200 1000}%
\special{fp}%
%
\special{pn 8}%
\special{pa 1400 1000}%
\special{pa 2200 1000}%
\special{fp}%
%
\special{pn 8}%
\special{pa 2800 1000}%
\special{pa 3800 1000}%
\special{fp}%
\put(8.0000,-6.0000){\makebox(0,0)[lb]{$D$}}%
\put(18.0000,-6.0000){\makebox(0,0)[lb]{$D^{\prime}$}}%
\put(4.0000,-9.0000){\makebox(0,0)[lb]{$\overline{p}$}}%
\put(12.0000,-9.0000){\makebox(0,0)[lb]{$p$}}%
\put(14.0000,-9.0000){\makebox(0,0)[lb]{$q$}}%
\put(22.0000,-9.0000){\makebox(0,0)[lb]{$\overline{q}$}}%
\put(28.0000,-9.0000){\makebox(0,0)[lb]{$\overline{p}$}}%
\put(38.0000,-9.0000){\makebox(0,0)[lb]{$\overline{q}$}}%
\put(28.0000,-6.0000){\makebox(0,0)[lb]{the added chord}}%
%
\special{pn 8}%
\special{pa 1400 1200}%
\special{pa 2200 1200}%
\special{fp}%
%
\special{pn 8}%
\special{pa 2800 1200}%
\special{pa 3800 1200}%
\special{fp}%
%
\special{pn 8}%
\special{pa 800 1000}%
\special{pa 700 950}%
\special{fp}%
\special{pa 800 1000}%
\special{pa 700 1040}%
\special{fp}%
%
\special{pn 8}%
\special{pa 1800 1000}%
\special{pa 1900 950}%
\special{fp}%
\special{pa 1800 1000}%
\special{pa 1900 1040}%
\special{fp}%
%
\special{pn 8}%
\special{pa 3300 1000}%
\special{pa 3400 950}%
\special{fp}%
\special{pa 3300 1000}%
\special{pa 3400 1040}%
\special{fp}%
%
\special{pn 8}%
\special{pa 400 1400}%
\special{pa 1200 1400}%
\special{fp}%
%
\special{pn 8}%
\special{pa 1400 1400}%
\special{pa 2200 1400}%
\special{fp}%
%
\special{pn 8}%
\special{pa 2800 1400}%
\special{pa 3800 1400}%
\special{fp}%
%
\special{pn 8}%
\special{pa 400 1600}%
\special{pa 1200 1600}%
\special{fp}%
%
\special{pn 8}%
\special{pa 1400 1600}%
\special{pa 2200 1600}%
\special{fp}%
%
\special{pn 8}%
\special{pa 2800 1600}%
\special{pa 3800 1600}%
\special{fp}%
%
\special{pn 8}%
\special{pa 800 1200}%
\special{pa 900 1150}%
\special{fp}%
\special{pa 800 1200}%
\special{pa 900 1240}%
\special{fp}%
%
\special{pn 8}%
\special{pa 800 1400}%
\special{pa 700 1350}%
\special{fp}%
\special{pa 800 1400}%
\special{pa 700 1440}%
\special{fp}%
%
\special{pn 8}%
\special{pa 800 1600}%
\special{pa 900 1550}%
\special{fp}%
\special{pa 800 1600}%
\special{pa 900 1640}%
\special{fp}%
%
\special{pn 8}%
\special{pa 1800 1600}%
\special{pa 1700 1550}%
\special{fp}%
\special{pa 1800 1600}%
\special{pa 1700 1640}%
\special{fp}%
%
\special{pn 8}%
\special{pa 1800 1400}%
\special{pa 1700 1350}%
\special{fp}%
\special{pa 1800 1400}%
\special{pa 1700 1440}%
\special{fp}%
%
\special{pn 8}%
\special{pa 1800 1200}%
\special{pa 1900 1150}%
\special{fp}%
\special{pa 1800 1200}%
\special{pa 1900 1240}%
\special{fp}%
%
\special{pn 8}%
\special{pa 3300 1200}%
\special{pa 3200 1150}%
\special{fp}%
\special{pa 3300 1200}%
\special{pa 3200 1240}%
\special{fp}%
%
\special{pn 8}%
\special{pa 3300 1400}%
\special{pa 3200 1350}%
\special{fp}%
\special{pa 3300 1400}%
\special{pa 3200 1440}%
\special{fp}%
%
\special{pn 8}%
\special{pa 3300 1600}%
\special{pa 3400 1550}%
\special{fp}%
\special{pa 3300 1600}%
\special{pa 3400 1640}%
\special{fp}%
%
\special{pn 20}%
\special{sh 1}%
\special{ar 400 1000 10 10 0  6.28318530717959E+0000}%
%
\special{pn 8}%
\special{pa 400 1200}%
\special{pa 1200 1200}%
\special{fp}%
%
\special{pn 20}%
\special{sh 1}%
\special{ar 400 1200 10 10 0  6.28318530717959E+0000}%
\special{sh 1}%
\special{ar 400 1400 10 10 0  6.28318530717959E+0000}%
\special{sh 1}%
\special{ar 400 1600 10 10 0  6.28318530717959E+0000}%
\special{sh 1}%
\special{ar 1200 1600 10 10 0  6.28318530717959E+0000}%
\special{sh 1}%
\special{ar 1200 1400 10 10 0  6.28318530717959E+0000}%
\special{sh 1}%
\special{ar 1200 1200 10 10 0  6.28318530717959E+0000}%
\special{sh 1}%
\special{ar 1200 1000 10 10 0  6.28318530717959E+0000}%
\special{sh 1}%
\special{ar 1400 1000 10 10 0  6.28318530717959E+0000}%
\special{sh 1}%
\special{ar 1400 1200 10 10 0  6.28318530717959E+0000}%
\special{sh 1}%
\special{ar 1400 1400 10 10 0  6.28318530717959E+0000}%
\special{sh 1}%
\special{ar 1400 1600 10 10 0  6.28318530717959E+0000}%
\special{sh 1}%
\special{ar 2200 1600 10 10 0  6.28318530717959E+0000}%
\special{sh 1}%
\special{ar 2200 1400 10 10 0  6.28318530717959E+0000}%
\special{sh 1}%
\special{ar 2200 1200 10 10 0  6.28318530717959E+0000}%
\special{sh 1}%
\special{ar 2200 1000 10 10 0  6.28318530717959E+0000}%
\special{sh 1}%
\special{ar 2800 1000 10 10 0  6.28318530717959E+0000}%
\special{sh 1}%
\special{ar 2800 1200 10 10 0  6.28318530717959E+0000}%
\special{sh 1}%
\special{ar 2800 1400 10 10 0  6.28318530717959E+0000}%
\special{sh 1}%
\special{ar 2800 1600 10 10 0  6.28318530717959E+0000}%
\special{sh 1}%
\special{ar 3800 1600 10 10 0  6.28318530717959E+0000}%
\special{sh 1}%
\special{ar 3800 1400 10 10 0  6.28318530717959E+0000}%
\special{sh 1}%
\special{ar 3800 1200 10 10 0  6.28318530717959E+0000}%
\special{sh 1}%
\special{ar 3800 1000 10 10 0  6.28318530717959E+0000}%
\end{picture}%
\end{center}
\end{figure}
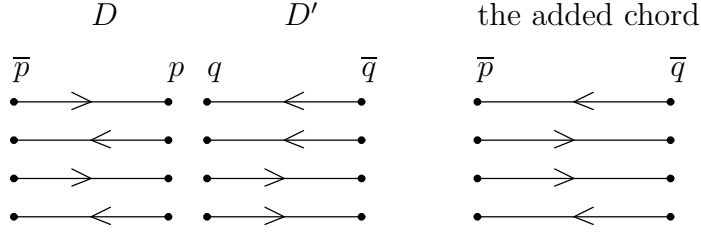

Finally, define $\mathcal{D}(D,p,D^{\prime},q)$ to be the closing
of the result of the second step. If $D$ has $m$ chords and $D^{\prime}$
has $m^{\prime}$ chords, then $\mathcal{D}(D,p,D^{\prime},q)$ has
$m+m^{\prime}-1$ chords. We have $\mathcal{D}(D,p,D^{\prime},q)
=N(C(D,p)\ast C(D^{\prime},q))$.
A schematic picture of this operation is as in Figure 6.

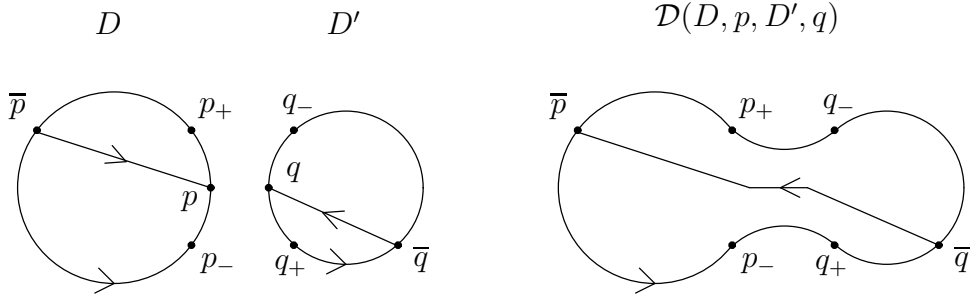
\begin{figure}
\begin{center}
\caption{the new labeled chord diagram $\mathcal{D}(D,p,D^{\prime},q)$}

\vspace{0.2cm}

\unitlength 0.1in
\begin{picture}( 49.4000, 15.2000)(  2.6000,-17.5000)
%
\special{pn 8}%
\special{ar 2000 1200 400 400  0.0000000 6.2831853}%
\put(7.0000,-4.0000){\makebox(0,0)[lb]{$D$}}%
\put(19.0000,-4.0000){\makebox(0,0)[lb]{$D^{\prime}$}}%
%
\special{pn 8}%
\special{ar 800 1200 500 500  0.0000000 6.2831853}%
%
\special{pn 20}%
\special{sh 1}%
\special{ar 1300 1200 10 10 0  6.28318530717959E+0000}%
%
\special{pn 20}%
\special{sh 1}%
\special{ar 1600 1200 10 10 0  6.28318530717959E+0000}%
%
\special{pn 20}%
\special{sh 1}%
\special{ar 1200 900 10 10 0  6.28318530717959E+0000}%
%
\special{pn 20}%
\special{sh 1}%
\special{ar 1730 1500 10 10 0  6.28318530717959E+0000}%
%
\special{pn 20}%
\special{sh 1}%
\special{ar 1200 1500 10 10 0  6.28318530717959E+0000}%
%
\special{pn 20}%
\special{sh 1}%
\special{ar 1730 900 10 10 0  6.28318530717959E+0000}%
%
\special{pn 20}%
\special{sh 1}%
\special{ar 400 900 10 10 0  6.28318530717959E+0000}%
%
\special{pn 20}%
\special{sh 1}%
\special{ar 2270 1500 10 10 0  6.28318530717959E+0000}%
%
\special{pn 8}%
\special{pa 1300 1200}%
\special{pa 400 910}%
\special{fp}%
%
\special{pn 8}%
\special{pa 1600 1200}%
\special{pa 2260 1500}%
\special{fp}%
\put(11.5000,-13.2000){\makebox(0,0)[lb]{$p$}}%
\put(16.9000,-11.7000){\makebox(0,0)[lb]{$q$}}%
\put(2.6000,-8.4000){\makebox(0,0)[lb]{$\overline{p}$}}%
\put(23.5000,-16.4000){\makebox(0,0)[lb]{$\overline{q}$}}%
\put(12.5000,-16.5000){\makebox(0,0)[lb]{$p_-$}}%
\put(12.4000,-8.4000){\makebox(0,0)[lb]{$p_+$}}%
\put(16.7000,-8.2000){\makebox(0,0)[lb]{$q_-$}}%
\put(16.3000,-16.6000){\makebox(0,0)[lb]{$q_+$}}%
%
\special{pn 20}%
\special{sh 1}%
\special{ar 4000 900 10 10 0  6.28318530717959E+0000}%
%
\special{pn 20}%
\special{sh 1}%
\special{ar 4530 1500 10 10 0  6.28318530717959E+0000}%
%
\special{pn 20}%
\special{sh 1}%
\special{ar 4000 1500 10 10 0  6.28318530717959E+0000}%
%
\special{pn 20}%
\special{sh 1}%
\special{ar 4530 900 10 10 0  6.28318530717959E+0000}%
%
\special{pn 20}%
\special{sh 1}%
\special{ar 3200 900 10 10 0  6.28318530717959E+0000}%
%
\special{pn 20}%
\special{sh 1}%
\special{ar 5070 1500 10 10 0  6.28318530717959E+0000}%
\put(30.6000,-8.4000){\makebox(0,0)[lb]{$\overline{p}$}}%
\put(51.5000,-16.4000){\makebox(0,0)[lb]{$\overline{q}$}}%
\put(40.5000,-16.5000){\makebox(0,0)[lb]{$p_-$}}%
\put(40.4000,-8.4000){\makebox(0,0)[lb]{$p_+$}}%
\put(44.7000,-8.2000){\makebox(0,0)[lb]{$q_-$}}%
\put(44.3000,-16.6000){\makebox(0,0)[lb]{$q_+$}}%
%
\special{pn 8}%
\special{ar 3600 1200 500 500  0.6435011 5.6396842}%
%
\special{pn 8}%
\special{ar 4800 1200 400 400  3.9269908 6.2831853}%
\special{ar 4800 1200 400 400  0.0000000 2.3561945}%
%
\special{pn 8}%
\special{ar 4270 1800 400 400  3.9958473 5.4071273}%
%
\special{pn 8}%
\special{ar 4270 600 400 400  0.8760581 2.2873380}%
%
\special{pn 8}%
\special{pa 3200 910}%
\special{pa 4090 1200}%
\special{fp}%
\special{pa 4090 1200}%
\special{pa 4390 1200}%
\special{fp}%
\special{pa 4390 1200}%
\special{pa 5070 1500}%
\special{fp}%
\put(36.0000,-4.0000){\makebox(0,0)[lb]{$\mathcal{D}(D,p,D^{\prime},q)$}}%
%
\special{pn 8}%
\special{pa 862 1064}%
\special{pa 790 980}%
\special{fp}%
\special{pa 862 1064}%
\special{pa 752 1072}%
\special{fp}%
%
\special{pn 8}%
\special{pa 1880 1330}%
\special{pa 1992 1322}%
\special{fp}%
\special{pa 1880 1330}%
\special{pa 1954 1414}%
\special{fp}%
%
\special{pn 8}%
\special{pa 4250 1200}%
\special{pa 4350 1150}%
\special{fp}%
\special{pa 4250 1200}%
\special{pa 4350 1250}%
\special{fp}%
%
\special{pn 8}%
\special{pa 800 1700}%
\special{pa 720 1620}%
\special{fp}%
\special{pa 800 1700}%
\special{pa 700 1750}%
\special{fp}%
%
\special{pn 8}%
\special{pa 3600 1700}%
\special{pa 3520 1620}%
\special{fp}%
\special{pa 3600 1700}%
\special{pa 3500 1750}%
\special{fp}%
%
\special{pn 8}%
\special{pa 2000 1600}%
\special{pa 1920 1520}%
\special{fp}%
\special{pa 2000 1600}%
\special{pa 1900 1650}%
\special{fp}%
\end{picture}%
\end{center}
\end{figure}

\begin{dfn}
\label{bracket-of-C}
Let $D$ and $D^{\prime}$ be labeled chord diagrams of
$m$ and $m^{\prime}$ chords, respectively. Set
$$[D,D^{\prime}]:=\sum_{(p,q)}\mathcal{D}(D,p,D^{\prime},q)
\in \mathcal{C}_{m+m^{\prime}-1},$$
where the sum is taken over all pairs of the vertices
of $D$ and $D^{\prime}$.
\end{dfn}

By construction, this formula is compatible with the formula
(\ref{N(C)-bracket}). Hence it defines a well-defined Lie algebra
structure on the space $\mathcal{C}$. But we continue
a diagrammatic argument for its own interest.
It is clear from the rule in Figure 5 that if $D_1$ is obtained
from $D$ by a single label change, then $[D,D^{\prime}]=-[D_1,D^{\prime}]$.
Therefore we can extend by linearity the bracket in Definition \ref{bracket-of-C}
to a $\mathbb{Q}$-linear map $[\ ,\ ]\colon \mathcal{C}\otimes \mathcal{C}\to \mathcal{C}$.

\begin{prop}
\label{C-is-Liealg}
The linear space $\mathcal{C}:=\prod_m \mathcal{C}_m$ has a structure of Lie algebra
with respect to the bracket defined above. Moreover, we have
$[\mathcal{C}_m,\mathcal{C}_{m^{\prime}}]\subset \mathcal{C}_{m+m^{\prime}-1}$.
\end{prop}

We call this Lie algebra {\it the Lie algebra of oriented chord diagrams}.

\begin{proof}
The anti-symmetry of the bracket is clear from the rule in Figure 5.
To prove the Jacobi identity, it suffices to show
\begin{equation}
\label{Jacobi-id}
[D,[D^{\prime},D^{\prime \prime}]]=[D^{\prime},[D,D^{\prime \prime}]]
+[[D,D^{\prime}],D^{\prime \prime}]
\end{equation}
for any labeled chord diagrams $D$, $D^{\prime}$, and $D^{\prime \prime}$.
Let $p$, $q$, and $r$ be vertices of $D$, $D^{\prime}$, and $D^{\prime \prime}$,
respectively. For simplicity, we denote
$\mathcal{D}(D^{\prime},q,D^{\prime \prime},r)=D^{\prime \prime \prime}$.
The contributions of $p$ to the
bracket $[D,D^{\prime \prime \prime}]$ consists of the diagrams
of the form $\mathcal{D}(D,p,D^{\prime \prime \prime},s)$, where
$s$ is a vertex of $D^{\prime \prime \prime}$.
We consider the following four cases.
\begin{enumerate}
\item[(1)] $s$ is the vertex corresponding to $\overline{q}$.
\item[(2)] $s$ is a vertex corresponding to some vertex of $D^{\prime}$ other than $\overline{q}$.
\item[(3)] $s$ is the vertex corresponding to $\overline{r}$.
\item[(4)] $s$ is a vertex corresponding to some vertex of $D^{\prime \prime}$
other than $\overline{r}$.
\end{enumerate}
See Figure 7. Consider the case (1). Then
$\mathcal{D}(D,p,D^{\prime \prime \prime},s)
=\mathcal{D}(D,p,D^{\prime \prime \prime},\overline{q})$.
But this is also equal to $\mathcal{D}(D_0,q,D^{\prime \prime},r)$,
where $D_0=\mathcal{D}(D,p,D^{\prime},\overline{q})$.
We can easily check the signs using Figure 5. Note that
$\mathcal{D}(D_0,q,D^{\prime \prime},r)$ will appear once when
we compute the second term of the right hand side of (\ref{Jacobi-id}).
The same thing happens to each contribution of the case (2).
By the same argument we see the cases (3) or (4) will appear
once at the first term of the right hand side of (\ref{Jacobi-id}).

\begin{figure}
\begin{center}
\caption{the four cases}

\vspace{0.2cm}

\unitlength 0.1in
\begin{picture}( 39.4000, 17.2000)(  3.0000,-19.5000)
\put(7.0000,-4.0000){\makebox(0,0)[lb]{$D$}}%
%
\special{pn 8}%
\special{ar 800 1400 500 500  0.0000000 6.2831853}%
%
\special{pn 20}%
\special{sh 1}%
\special{ar 1300 1400 10 10 0  6.28318530717959E+0000}%
%
\special{pn 20}%
\special{sh 1}%
\special{ar 1200 1100 10 10 0  6.28318530717959E+0000}%
%
\special{pn 20}%
\special{sh 1}%
\special{ar 1200 1700 10 10 0  6.28318530717959E+0000}%
%
\special{pn 20}%
\special{sh 1}%
\special{ar 400 1100 10 10 0  6.28318530717959E+0000}%
%
\special{pn 8}%
\special{pa 1300 1400}%
\special{pa 400 1110}%
\special{fp}%
\put(11.5000,-15.2000){\makebox(0,0)[lb]{$p$}}%
%
\special{pn 20}%
\special{sh 1}%
\special{ar 3000 1100 10 10 0  6.28318530717959E+0000}%
%
\special{pn 20}%
\special{sh 1}%
\special{ar 3530 1700 10 10 0  6.28318530717959E+0000}%
%
\special{pn 20}%
\special{sh 1}%
\special{ar 3000 1700 10 10 0  6.28318530717959E+0000}%
%
\special{pn 20}%
\special{sh 1}%
\special{ar 3530 1100 10 10 0  6.28318530717959E+0000}%
%
\special{pn 20}%
\special{sh 1}%
\special{ar 2200 1100 10 10 0  6.28318530717959E+0000}%
%
\special{pn 20}%
\special{sh 1}%
\special{ar 4070 1700 10 10 0  6.28318530717959E+0000}%
\put(21.1000,-10.1000){\makebox(0,0)[lb]{$\overline{q}$}}%
\put(40.5000,-19.0000){\makebox(0,0)[lb]{$\overline{r}$}}%
%
\special{pn 8}%
\special{ar 2600 1400 500 500  0.6435011 5.6396842}%
%
\special{pn 8}%
\special{ar 3800 1400 400 400  3.9269908 6.2831853}%
\special{ar 3800 1400 400 400  0.0000000 2.3561945}%
%
\special{pn 8}%
\special{ar 3270 2000 400 400  3.9958473 5.4071273}%
%
\special{pn 8}%
\special{ar 3270 800 400 400  0.8760581 2.2873380}%
%
\special{pn 8}%
\special{pa 2200 1110}%
\special{pa 3090 1400}%
\special{fp}%
\special{pa 3090 1400}%
\special{pa 3390 1400}%
\special{fp}%
\special{pa 3390 1400}%
\special{pa 4070 1700}%
\special{fp}%
\put(31.0000,-4.0000){\makebox(0,0)[lb]{$D^{\prime \prime \prime}$}}%
%
\special{pn 8}%
\special{pa 862 1264}%
\special{pa 790 1180}%
\special{fp}%
\special{pa 862 1264}%
\special{pa 752 1272}%
\special{fp}%
%
\special{pn 8}%
\special{pa 3250 1400}%
\special{pa 3350 1350}%
\special{fp}%
\special{pa 3250 1400}%
\special{pa 3350 1450}%
\special{fp}%
\put(24.0000,-7.0000){\makebox(0,0)[lb]{$D^{\prime}$}}%
\put(37.0000,-7.0000){\makebox(0,0)[lb]{$D^{\prime \prime}$}}%
%
\special{pn 20}%
\special{sh 1}%
\special{ar 2310 1800 10 10 0  6.28318530717959E+0000}%
%
\special{pn 20}%
\special{sh 1}%
\special{ar 4150 1200 10 10 0  6.28318530717959E+0000}%
\put(30.5000,-18.1000){\makebox(0,0)[lb]{$q_-$}}%
\put(30.5000,-10.6000){\makebox(0,0)[lb]{$q_+$}}%
\put(34.3000,-10.5000){\makebox(0,0)[lb]{$r_-$}}%
\put(34.1000,-18.5000){\makebox(0,0)[lb]{$r_+$}}%
\put(15.6000,-12.5000){\makebox(0,0)[lb]{\fbox{case 1}}}%
\put(17.1000,-20.0000){\makebox(0,0)[lb]{\fbox{case 2}}}%
\put(42.1000,-17.6000){\makebox(0,0)[lb]{\fbox{case 3}}}%
\put(42.4000,-11.7000){\makebox(0,0)[lb]{\fbox{case 4}}}%
%
\special{pn 8}%
\special{pa 800 1900}%
\special{pa 720 1820}%
\special{fp}%
\special{pa 800 1900}%
\special{pa 700 1950}%
\special{fp}%
%
\special{pn 8}%
\special{pa 2600 1900}%
\special{pa 2520 1820}%
\special{fp}%
\special{pa 2600 1900}%
\special{pa 2500 1950}%
\special{fp}%
\end{picture}%
\end{center}
\end{figure}
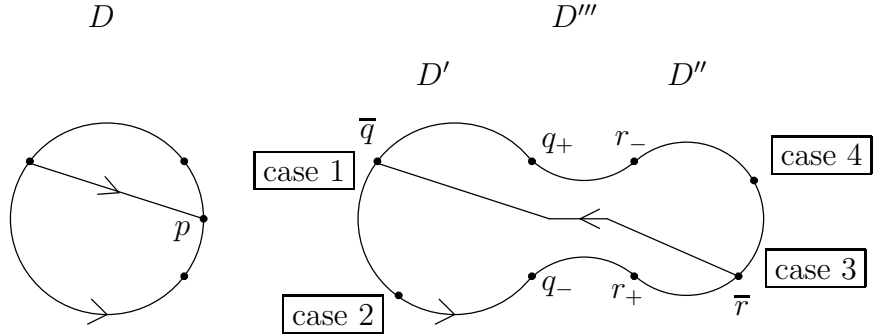

Now we consider the contributions
$\mathcal{D}(D,p,D^{\prime \prime \prime},s)$ for all $p$, $q$, and $r$,
and subtract them from the right hand side of (\ref{Jacobi-id}).
The remaining terms consist of two types. One comes from the
first term, and is written as $\mathcal{D}(D^{\prime},q,D_1,t)$,
where $D_1=\mathcal{D}(D,p,D^{\prime \prime},r)$ and $t$ is a vertex
corresponding to some vertex of $D$. The other comes from the
second term, and is written as $\mathcal{D}(D_2,u,D^{\prime \prime},r)$,
where $D_2=\mathcal{D}(D,p,D^{\prime},q)$ and $u$ is a vertex
corresponding to some vertex of $D$.
By the same argument as before, we can see these two types cancel.
This proves the Jacobi identity, hence completes the proof.
\end{proof}

An {\it isolated chord} in a labeled chord diagram is
a labeled chord whose two ends are adjacent on the core.

\begin{lem}
\label{cancel}
Let $D^{\prime}$ be a labeled chord diagram having an
isolated chord. Let $q_0$, $q_1$ be the ends of the isolated
chord. Then
$$\sum_p \mathcal{D}(D,p,D^{\prime},q_0)
+\sum_p \mathcal{D}(D,p,D^{\prime},q_1)=0$$
for any labeled chord diagram $D$. Here the sums are taken over
all the vertices of $D$.
\end{lem}

\begin{proof}
We may assume $q_1$ is next to $q_0$ with respect to
the orientation of the core. Let $p_0$ be a vertex of $D$
and $p_1$ the vertex next to $p_0$. Then we have
$$\mathcal{D}(D,p_0,D^{\prime},q_0)=-\mathcal{D}(D,p_1,D^{\prime},q_1).$$
This proves the lemma.
\end{proof}

For $m\ge 1$, let $\Omega_m\in \mathcal{C}_m$ be the closing of the
labeled linear chord diagram $I_m=\{ (1,2),(3,4),\ldots, (2m-1,2m)\}$.
We say a vertex of $\Omega_m$ is {\it odd} (resp.\ {\it even}) if it
corresponds to an odd (resp.\ even) numbered vertex in $I_m$.
All the chords of $\Omega_m$ are isolated, hence by Lemma \ref{cancel},
$[D,\Omega_m]=0$ for any labeled chord diagram $D$. Therefore,
$\Omega_m\in Z(\mathcal{C})$.

For integers $a,b\ge 1$, define a labeled chord diagram $D(a,b)$ to
be the closing of the labeled linear chord diagram
$$\{ (1,2),(3,4),\ldots,(2a-1,2a),(2a+1,2a+2b+2),(2a+2,2a+3),
\ldots,(2a+2b,2a+2b+1) \}.$$

\begin{figure}
\begin{center}
\caption{$D(1,3)$}

\vspace{0.2cm}

\unitlength 0.1in
\begin{picture}( 16.0000, 16.0000)(  4.5000,-17.6000)
%
\special{pn 8}%
\special{ar 1250 960 800 800  0.0000000 6.2831853}%
%
\special{pn 8}%
\special{pa 1352 1750}%
\special{pa 1340 1720}%
\special{pa 1332 1688}%
\special{pa 1324 1658}%
\special{pa 1318 1626}%
\special{pa 1314 1594}%
\special{pa 1310 1562}%
\special{pa 1310 1530}%
\special{pa 1310 1498}%
\special{pa 1314 1466}%
\special{pa 1320 1436}%
\special{pa 1330 1406}%
\special{pa 1346 1378}%
\special{pa 1366 1352}%
\special{pa 1392 1334}%
\special{pa 1424 1326}%
\special{pa 1454 1326}%
\special{pa 1486 1334}%
\special{pa 1516 1346}%
\special{pa 1542 1362}%
\special{pa 1568 1384}%
\special{pa 1590 1404}%
\special{pa 1612 1428}%
\special{pa 1634 1452}%
\special{pa 1652 1478}%
\special{pa 1670 1506}%
\special{pa 1686 1532}%
\special{pa 1702 1560}%
\special{pa 1716 1590}%
\special{pa 1720 1602}%
\special{sp}%
%
\special{pn 8}%
\special{pa 1148 172}%
\special{pa 1160 202}%
\special{pa 1168 232}%
\special{pa 1176 264}%
\special{pa 1182 294}%
\special{pa 1186 326}%
\special{pa 1190 358}%
\special{pa 1190 390}%
\special{pa 1190 422}%
\special{pa 1186 454}%
\special{pa 1180 486}%
\special{pa 1170 516}%
\special{pa 1154 544}%
\special{pa 1134 568}%
\special{pa 1108 586}%
\special{pa 1078 596}%
\special{pa 1046 596}%
\special{pa 1014 588}%
\special{pa 986 574}%
\special{pa 958 558}%
\special{pa 932 538}%
\special{pa 910 516}%
\special{pa 888 492}%
\special{pa 868 468}%
\special{pa 848 442}%
\special{pa 830 416}%
\special{pa 814 388}%
\special{pa 798 360}%
\special{pa 784 332}%
\special{pa 780 320}%
\special{sp}%
%
\special{pn 8}%
\special{pa 2040 860}%
\special{pa 2010 872}%
\special{pa 1978 880}%
\special{pa 1948 886}%
\special{pa 1916 892}%
\special{pa 1884 898}%
\special{pa 1852 900}%
\special{pa 1820 900}%
\special{pa 1788 900}%
\special{pa 1756 896}%
\special{pa 1726 890}%
\special{pa 1696 878}%
\special{pa 1668 864}%
\special{pa 1642 844}%
\special{pa 1624 816}%
\special{pa 1616 786}%
\special{pa 1616 754}%
\special{pa 1624 724}%
\special{pa 1638 694}%
\special{pa 1654 668}%
\special{pa 1674 642}%
\special{pa 1696 618}%
\special{pa 1720 598}%
\special{pa 1744 576}%
\special{pa 1770 558}%
\special{pa 1796 540}%
\special{pa 1824 524}%
\special{pa 1852 510}%
\special{pa 1882 496}%
\special{pa 1894 490}%
\special{sp}%
%
\special{pn 8}%
\special{pa 612 1432}%
\special{pa 640 1418}%
\special{pa 668 1404}%
\special{pa 696 1386}%
\special{pa 724 1370}%
\special{pa 750 1352}%
\special{pa 774 1332}%
\special{pa 798 1310}%
\special{pa 822 1288}%
\special{pa 840 1262}%
\special{pa 860 1236}%
\special{pa 872 1208}%
\special{pa 884 1176}%
\special{pa 886 1146}%
\special{pa 882 1114}%
\special{pa 866 1086}%
\special{pa 844 1062}%
\special{pa 818 1046}%
\special{pa 788 1034}%
\special{pa 758 1026}%
\special{pa 726 1022}%
\special{pa 694 1020}%
\special{pa 662 1020}%
\special{pa 630 1022}%
\special{pa 598 1026}%
\special{pa 566 1032}%
\special{pa 536 1040}%
\special{pa 504 1048}%
\special{pa 474 1058}%
\special{pa 462 1062}%
\special{sp}%
%
\special{pn 8}%
\special{pa 588 522}%
\special{pa 616 536}%
\special{pa 646 550}%
\special{pa 674 564}%
\special{pa 700 582}%
\special{pa 728 598}%
\special{pa 754 616}%
\special{pa 778 638}%
\special{pa 802 658}%
\special{pa 826 680}%
\special{pa 850 702}%
\special{pa 872 726}%
\special{pa 892 750}%
\special{pa 914 774}%
\special{pa 934 800}%
\special{pa 952 826}%
\special{pa 970 852}%
\special{pa 986 878}%
\special{pa 1002 906}%
\special{pa 1018 934}%
\special{pa 1034 962}%
\special{pa 1046 992}%
\special{pa 1058 1022}%
\special{pa 1070 1052}%
\special{pa 1082 1082}%
\special{pa 1092 1112}%
\special{pa 1100 1142}%
\special{pa 1108 1174}%
\special{pa 1116 1204}%
\special{pa 1122 1236}%
\special{pa 1128 1268}%
\special{pa 1132 1298}%
\special{pa 1136 1330}%
\special{pa 1138 1362}%
\special{pa 1138 1394}%
\special{pa 1138 1426}%
\special{pa 1138 1458}%
\special{pa 1134 1490}%
\special{pa 1132 1522}%
\special{pa 1126 1554}%
\special{pa 1122 1586}%
\special{pa 1114 1616}%
\special{pa 1104 1648}%
\special{pa 1096 1678}%
\special{pa 1084 1708}%
\special{pa 1070 1736}%
\special{pa 1068 1744}%
\special{sp}%
%
\special{pn 8}%
\special{sh 1}%
\special{ar 460 1060 10 10 0  6.28318530717959E+0000}%
%
\special{pn 20}%
\special{sh 1}%
\special{ar 460 1060 10 10 0  6.28318530717959E+0000}%
%
\special{pn 20}%
\special{sh 1}%
\special{ar 610 1430 10 10 0  6.28318530717959E+0000}%
%
\special{pn 20}%
\special{sh 1}%
\special{ar 1070 1740 10 10 0  6.28318530717959E+0000}%
%
\special{pn 20}%
\special{sh 1}%
\special{ar 1350 1750 10 10 0  6.28318530717959E+0000}%
%
\special{pn 20}%
\special{sh 1}%
\special{ar 1720 1600 10 10 0  6.28318530717959E+0000}%
%
\special{pn 20}%
\special{sh 1}%
\special{ar 2040 860 10 10 0  6.28318530717959E+0000}%
%
\special{pn 20}%
\special{sh 1}%
\special{ar 1890 500 10 10 0  6.28318530717959E+0000}%
%
\special{pn 20}%
\special{sh 1}%
\special{ar 1150 180 10 10 0  6.28318530717959E+0000}%
%
\special{pn 20}%
\special{sh 1}%
\special{ar 780 320 10 10 0  6.28318530717959E+0000}%
%
\special{pn 20}%
\special{sh 1}%
\special{ar 590 530 10 10 0  6.28318530717959E+0000}%
%
\special{pn 8}%
\special{pa 1960 1340}%
\special{pa 1854 1376}%
\special{fp}%
\special{pa 1960 1340}%
\special{pa 1960 1452}%
\special{fp}%
%
\special{pn 8}%
\special{pa 1190 360}%
\special{pa 1230 260}%
\special{fp}%
\special{pa 1190 360}%
\special{pa 1112 284}%
\special{fp}%
%
\special{pn 8}%
\special{pa 1630 1460}%
\special{pa 1620 1352}%
\special{fp}%
\special{pa 1630 1460}%
\special{pa 1526 1428}%
\special{fp}%
%
\special{pn 8}%
\special{pa 1020 930}%
\special{pa 1002 1038}%
\special{fp}%
\special{pa 1020 930}%
\special{pa 1112 988}%
\special{fp}%
%
\special{pn 8}%
\special{pa 730 1370}%
\special{pa 836 1344}%
\special{fp}%
\special{pa 730 1370}%
\special{pa 746 1264}%
\special{fp}%
%
\special{pn 8}%
\special{pa 1790 540}%
\special{pa 1684 558}%
\special{fp}%
\special{pa 1790 540}%
\special{pa 1766 646}%
\special{fp}%
\put(4.5000,-5.0000){\makebox(0,0)[lb]{$\overline{\delta}$}}%
\put(8.6000,-18.6000){\makebox(0,0)[lb]{$\delta$}}%
\end{picture}%
\end{center}
\end{figure}

If $a\neq b$,$D(a,b)$ is of maximal index since it has a unique
non-isolated chord dividing the vertices not touching the chord
into $2a$ and $2b$ vertices. Also we have $D(b,a)=-D(a,b)\in \mathcal{C}$,
in particular $D(a,a)=0$.
We denote by $\delta$ and $\overline{\delta}$, the vertices
corresponding to $2a+1$ and $2a+2b+2$, respectively. See Figure 8.
By Lemma \ref{cancel}, for any labeled chord diagram $D$, we have
\begin{equation}
\label{[D,D(a,b)]}
[D,D(a,b)]=\sum_p \mathcal{D}(D,p,D(a,b),\delta)+
\sum_p \mathcal{D}(D,p,D(a,b),\overline{\delta}).
\end{equation}
We shall look into each term in more detail.
The diagram $\mathcal{D}(D,p,D(a,b),\delta)$ is obtained
from $D$ by inserting $b$ isolated chords between $p_-$ and $p$,
and $a$ isolated chords between $p$ and $p_+$.
Similarly the diagram $\mathcal{D}(D,p,D(a,b),\overline{\delta})$ is
obtained from $D$ by inserting $a$ isolated chords between
$p_-$ and $p$, and $b$ isolated chords between $p$ and $p_+$,
and reversing the orientation of the chord through $p$.
Figure 10 is a picture of the results. Here,
for simplicity we write a sequence of $n$ isolated chords as Figure 9.

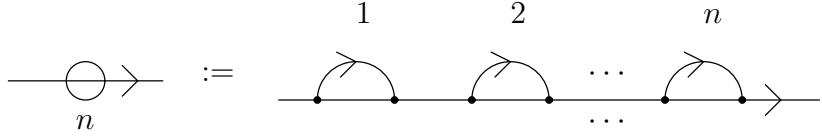
\begin{figure}
\begin{center}
\caption{$n$ isolated chords}

\vspace{0.2cm}

\unitlength 0.1in
\begin{picture}( 42.0000,  6.5000)(  2.0000, -8.8000)
%
\special{pn 20}%
\special{sh 1}%
\special{ar 1800 800 10 10 0  6.28318530717959E+0000}%
%
\special{pn 20}%
\special{sh 1}%
\special{ar 2200 800 10 10 0  6.28318530717959E+0000}%
%
\special{pn 20}%
\special{sh 1}%
\special{ar 2600 800 10 10 0  6.28318530717959E+0000}%
%
\special{pn 20}%
\special{sh 1}%
\special{ar 3000 800 10 10 0  6.28318530717959E+0000}%
%
\special{pn 20}%
\special{sh 1}%
\special{ar 3600 800 10 10 0  6.28318530717959E+0000}%
%
\special{pn 20}%
\special{sh 1}%
\special{ar 4000 800 10 10 0  6.28318530717959E+0000}%
%
\special{pn 8}%
\special{pa 1600 800}%
\special{pa 4400 800}%
\special{fp}%
\put(32.0000,-9.5000){\makebox(0,0)[lb]{$\cdots$}}%
%
\special{pn 8}%
\special{pa 4200 800}%
\special{pa 4120 720}%
\special{fp}%
\special{pa 4200 800}%
\special{pa 4120 880}%
\special{fp}%
%
\special{pn 8}%
\special{ar 2000 800 200 200  3.1415927 6.2831853}%
%
\special{pn 8}%
\special{ar 2800 800 200 200  3.1415927 6.2831853}%
%
\special{pn 8}%
\special{ar 3800 800 200 200  3.1415927 6.2831853}%
%
\special{pn 8}%
\special{pa 2000 600}%
\special{pa 1900 550}%
\special{fp}%
%
\special{pn 8}%
\special{pa 2000 600}%
\special{pa 1920 680}%
\special{fp}%
%
\special{pn 8}%
\special{pa 2800 600}%
\special{pa 2700 550}%
\special{fp}%
%
\special{pn 8}%
\special{pa 2800 600}%
\special{pa 2720 680}%
\special{fp}%
%
\special{pn 8}%
\special{pa 3800 600}%
\special{pa 3700 550}%
\special{fp}%
%
\special{pn 8}%
\special{pa 3800 600}%
\special{pa 3720 680}%
\special{fp}%
\put(20.0000,-4.0000){\makebox(0,0)[lb]{1}}%
\put(28.0000,-4.0000){\makebox(0,0)[lb]{2}}%
\put(38.0000,-4.0000){\makebox(0,0)[lb]{$n$}}%
%
\special{pn 8}%
\special{ar 600 700 100 100  0.0000000 6.2831853}%
%
\special{pn 8}%
\special{pa 870 700}%
\special{pa 790 620}%
\special{fp}%
\special{pa 870 700}%
\special{pa 790 780}%
\special{fp}%
\put(5.5000,-9.5000){\makebox(0,0)[lb]{$n$}}%
\put(12.0000,-7.0000){\makebox(0,0)[lb]{$:=$}}%
%
\special{pn 8}%
\special{pa 200 700}%
\special{pa 1000 700}%
\special{fp}%
\put(32.0000,-7.0000){\makebox(0,0)[lb]{$\cdots$}}%
\end{picture}%
\end{center}
\end{figure}

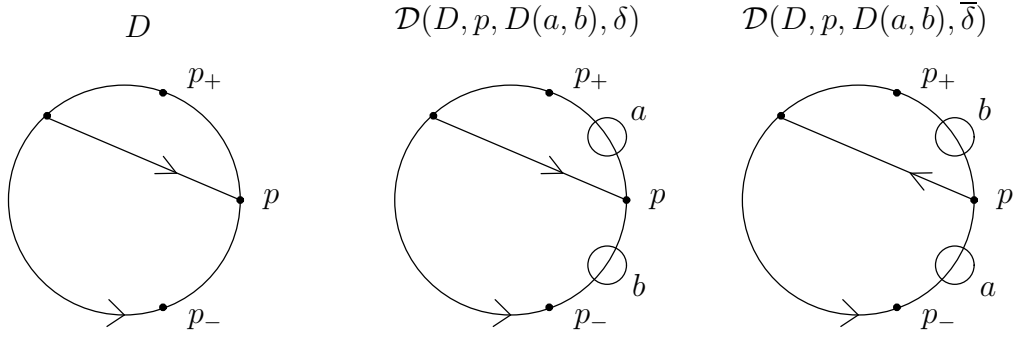
\begin{figure}
\begin{center}
\caption{Inserting isolated chords}

\vspace{0.2cm}

\unitlength 0.1in
\begin{picture}( 51.2000, 16.8000)(  4.0000,-20.6000)
%
\special{pn 8}%
\special{ar 1000 1400 600 600  0.0000000 6.2831853}%
%
\special{pn 8}%
\special{ar 3000 1400 600 600  0.0000000 6.2831853}%
%
\special{pn 8}%
\special{ar 4800 1400 600 600  0.0000000 6.2831853}%
%
\special{pn 20}%
\special{sh 1}%
\special{ar 1600 1400 10 10 0  6.28318530717959E+0000}%
%
\special{pn 20}%
\special{sh 1}%
\special{ar 3600 1400 10 10 0  6.28318530717959E+0000}%
%
\special{pn 20}%
\special{sh 1}%
\special{ar 5400 1400 10 10 0  6.28318530717959E+0000}%
%
\special{pn 20}%
\special{sh 1}%
\special{ar 600 960 10 10 0  6.28318530717959E+0000}%
%
\special{pn 20}%
\special{sh 1}%
\special{ar 2600 960 10 10 0  6.28318530717959E+0000}%
%
\special{pn 20}%
\special{sh 1}%
\special{ar 4400 960 10 10 0  6.28318530717959E+0000}%
%
\special{pn 20}%
\special{sh 1}%
\special{ar 1200 840 10 10 0  6.28318530717959E+0000}%
%
\special{pn 20}%
\special{sh 1}%
\special{ar 3200 840 10 10 0  6.28318530717959E+0000}%
%
\special{pn 20}%
\special{sh 1}%
\special{ar 5000 840 10 10 0  6.28318530717959E+0000}%
%
\special{pn 20}%
\special{sh 1}%
\special{ar 5000 1960 10 10 0  6.28318530717959E+0000}%
%
\special{pn 20}%
\special{sh 1}%
\special{ar 3200 1960 10 10 0  6.28318530717959E+0000}%
%
\special{pn 20}%
\special{sh 1}%
\special{ar 1200 1960 10 10 0  6.28318530717959E+0000}%
%
\special{pn 8}%
\special{pa 1000 2000}%
\special{pa 920 1920}%
\special{fp}%
\special{pa 1000 2000}%
\special{pa 910 2060}%
\special{fp}%
%
\special{pn 8}%
\special{pa 3000 2000}%
\special{pa 2920 1920}%
\special{fp}%
\special{pa 3000 2000}%
\special{pa 2910 2060}%
\special{fp}%
%
\special{pn 8}%
\special{pa 4800 2000}%
\special{pa 4720 1920}%
\special{fp}%
\special{pa 4800 2000}%
\special{pa 4710 2060}%
\special{fp}%
%
\special{pn 8}%
\special{pa 1600 1400}%
\special{pa 600 970}%
\special{fp}%
%
\special{pn 8}%
\special{pa 3600 1400}%
\special{pa 2600 970}%
\special{fp}%
%
\special{pn 8}%
\special{pa 5400 1400}%
\special{pa 4400 970}%
\special{fp}%
\put(10.0000,-5.5000){\makebox(0,0)[lb]{$D$}}%
\put(24.0000,-5.5000){\makebox(0,0)[lb]{$\mathcal{D}(D,p,D(a,b),\delta)$}}%
\put(42.0000,-5.5000){\makebox(0,0)[lb]{$\mathcal{D}(D,p,D(a,b),\overline{\delta})$}}%
%
\special{pn 8}%
\special{ar 3500 1740 100 100  0.0000000 6.2831853}%
%
\special{pn 8}%
\special{ar 5300 1740 100 100  0.0000000 6.2831853}%
%
\special{pn 8}%
\special{ar 5300 1070 100 100  0.0000000 6.2831853}%
%
\special{pn 8}%
\special{ar 3500 1070 100 100  0.0000000 6.2831853}%
\put(36.3000,-19.0000){\makebox(0,0)[lb]{$b$}}%
\put(36.2000,-9.8000){\makebox(0,0)[lb]{$a$}}%
\put(54.2000,-9.8000){\makebox(0,0)[lb]{$b$}}%
\put(54.3000,-19.0000){\makebox(0,0)[lb]{$a$}}%
\put(17.2000,-14.5000){\makebox(0,0)[lb]{$p$}}%
\put(37.2000,-14.5000){\makebox(0,0)[lb]{$p$}}%
\put(55.2000,-14.5000){\makebox(0,0)[lb]{$p$}}%
\put(13.3000,-20.8000){\makebox(0,0)[lb]{$p_-$}}%
\put(33.3000,-20.8000){\makebox(0,0)[lb]{$p_-$}}%
\put(51.3000,-20.8000){\makebox(0,0)[lb]{$p_-$}}%
\put(13.3000,-8.2000){\makebox(0,0)[lb]{$p_+$}}%
\put(33.3000,-8.2000){\makebox(0,0)[lb]{$p_+$}}%
\put(51.3000,-8.2000){\makebox(0,0)[lb]{$p_+$}}%
%
\special{pn 8}%
\special{pa 1270 1260}%
\special{pa 1202 1172}%
\special{fp}%
\special{pa 1270 1260}%
\special{pa 1158 1262}%
\special{fp}%
%
\special{pn 8}%
\special{pa 3270 1260}%
\special{pa 3202 1172}%
\special{fp}%
\special{pa 3270 1260}%
\special{pa 3158 1262}%
\special{fp}%
%
\special{pn 8}%
\special{pa 5070 1260}%
\special{pa 5140 1348}%
\special{fp}%
\special{pa 5070 1260}%
\special{pa 5182 1258}%
\special{fp}%
\end{picture}%
\end{center}
\end{figure}

For a while, fix $m\ge1$ and let $a=m$, $b=2m+1$. The following
two lemmas are the key in the sequel.

\begin{lem}
\label{hodai1}
Let $D_1$ and $D_2$ be labeled chord diagrams of $m$ chords,
and let $p_1$ and $p_2$ be vertices of $D_1$ and $D_2$, respectively.
Suppose $\mathcal{D}(D_1,p_1,D(a,b),\delta)=
\pm \mathcal{D}(D_2,p_2,D(a,b),\delta)$ or
$\mathcal{D}(D_1,p_1,D(a,b),\overline{\delta})=
\pm \mathcal{D}(D_2,p_2,D(a,b),\overline{\delta})$, in $\mathcal{C}$.
Then the cuts $C(D_1,p_1)$ and $C(D_2,p_2)$ are equal in $\mathcal{LC}_m$
up to sign. In particular, $D_1=\pm D_2\in \mathcal{C}_m$.
\end{lem}

\begin{proof}
We only consider the case $\mathcal{D}(D_1,p_1,D(a,b),\delta)=
\pm \mathcal{D}(D_2,p_2,D(a,b),\delta)$.
We draw a picture of $C(D_i,p_i)$ as Figure 11.
Here $D_i^{\prime}$ is the part of $C(D_i,p_i)$ between $p_i$ and
$\overline{p_i}$, $D_i^{\prime \prime}$ the part on the right of
$\overline{p_i}$, and the dotted line indicates the chords
connecting the vertices in $D_i^{\prime}$ and $D_i^{\prime \prime}$.
Then the diagrams $\mathcal{D}(D_i,p_i,D(a,b),\delta)$,
$i=1,2$ look like Figure 12.

\begin{figure}
\begin{center}
\caption{the cut $C(D_i,p_i)$}

\vspace{0.2cm}

\unitlength 0.1in
\begin{picture}( 26.0000,  7.5000)(  4.0000, -9.5000)
%
\special{pn 8}%
\special{pa 400 800}%
\special{pa 3000 800}%
\special{fp}%
%
\special{pn 20}%
\special{sh 1}%
\special{ar 600 800 10 10 0  6.28318530717959E+0000}%
%
\special{pn 8}%
\special{pa 800 900}%
\special{pa 1600 900}%
\special{pa 1600 700}%
\special{pa 800 700}%
\special{pa 800 900}%
\special{fp}%
%
\special{pn 8}%
\special{pa 2000 900}%
\special{pa 2600 900}%
\special{pa 2600 700}%
\special{pa 2000 700}%
\special{pa 2000 900}%
\special{fp}%
%
\special{pn 20}%
\special{sh 1}%
\special{ar 1800 800 10 10 0  6.28318530717959E+0000}%
%
\special{pn 8}%
\special{ar 1200 800 600 600  3.1415927 6.2831853}%
%
\special{pn 8}%
\special{pa 2800 800}%
\special{pa 2720 720}%
\special{fp}%
\special{pa 2800 800}%
\special{pa 2720 880}%
\special{fp}%
%
\special{pn 8}%
\special{ar 1800 800 400 400  3.1415927 3.2915927}%
\special{ar 1800 800 400 400  3.3815927 3.5315927}%
\special{ar 1800 800 400 400  3.6215927 3.7715927}%
\special{ar 1800 800 400 400  3.8615927 4.0115927}%
\special{ar 1800 800 400 400  4.1015927 4.2515927}%
\special{ar 1800 800 400 400  4.3415927 4.4915927}%
\special{ar 1800 800 400 400  4.5815927 4.7315927}%
\special{ar 1800 800 400 400  4.8215927 4.9715927}%
\special{ar 1800 800 400 400  5.0615927 5.2115927}%
\special{ar 1800 800 400 400  5.3015927 5.4515927}%
\special{ar 1800 800 400 400  5.5415927 5.6915927}%
\special{ar 1800 800 400 400  5.7815927 5.9315927}%
\special{ar 1800 800 400 400  6.0215927 6.1715927}%
\special{ar 1800 800 400 400  6.2615927 6.2831853}%
\put(5.5000,-10.0000){\makebox(0,0)[lb]{$p_i$}}%
\put(17.5000,-10.0000){\makebox(0,0)[lb]{$\overline{p_i}$}}%
\put(11.0000,-11.2000){\makebox(0,0)[lb]{$D_i^{\prime}$}}%
\put(22.0000,-11.2000){\makebox(0,0)[lb]{$D_i^{\prime \prime}$}}%
\end{picture}%
\end{center}
\end{figure}
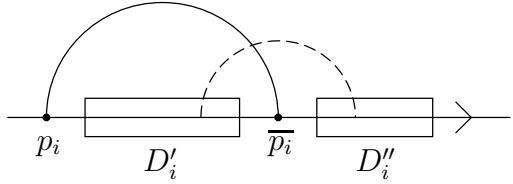

\begin{figure}
\begin{center}
\caption{the two diagrams $\mathcal{D}(D_i,p_i,D(a,b),\delta)$}

\vspace{0.2cm}

\unitlength 0.1in
\begin{picture}( 37.8000, 20.3500)(  3.1000,-23.3500)
%
\special{pn 8}%
\special{ar 1190 1600 600 600  0.0000000 6.2831853}%
%
\special{pn 8}%
\special{pa 1760 1376}%
\special{pa 1736 1398}%
\special{pa 1712 1418}%
\special{pa 1688 1440}%
\special{pa 1662 1458}%
\special{pa 1634 1474}%
\special{pa 1608 1492}%
\special{pa 1580 1508}%
\special{pa 1550 1520}%
\special{pa 1522 1534}%
\special{pa 1492 1548}%
\special{pa 1462 1558}%
\special{pa 1432 1568}%
\special{pa 1400 1576}%
\special{pa 1370 1582}%
\special{pa 1338 1588}%
\special{pa 1306 1592}%
\special{pa 1274 1598}%
\special{pa 1242 1600}%
\special{pa 1210 1600}%
\special{pa 1178 1600}%
\special{pa 1148 1598}%
\special{pa 1116 1596}%
\special{pa 1084 1592}%
\special{pa 1052 1586}%
\special{pa 1020 1580}%
\special{pa 990 1570}%
\special{pa 960 1562}%
\special{pa 928 1552}%
\special{pa 898 1540}%
\special{pa 870 1526}%
\special{pa 840 1514}%
\special{pa 812 1498}%
\special{pa 786 1480}%
\special{pa 758 1466}%
\special{pa 732 1448}%
\special{pa 706 1426}%
\special{pa 682 1406}%
\special{pa 658 1384}%
\special{pa 636 1362}%
\special{pa 632 1360}%
\special{sp}%
%
\special{pn 20}%
\special{sh 1}%
\special{ar 644 1362 10 10 0  6.28318530717959E+0000}%
%
\special{pn 20}%
\special{sh 1}%
\special{ar 1744 1390 10 10 0  6.28318530717959E+0000}%
%
\special{pn 8}%
\special{ar 3420 1600 600 600  0.0000000 6.2831853}%
%
\special{pn 8}%
\special{pa 4000 1436}%
\special{pa 3976 1458}%
\special{pa 3952 1478}%
\special{pa 3928 1498}%
\special{pa 3900 1516}%
\special{pa 3874 1534}%
\special{pa 3846 1550}%
\special{pa 3818 1566}%
\special{pa 3790 1580}%
\special{pa 3760 1590}%
\special{pa 3730 1600}%
\special{pa 3698 1610}%
\special{pa 3668 1616}%
\special{pa 3636 1622}%
\special{pa 3604 1624}%
\special{pa 3572 1626}%
\special{pa 3540 1624}%
\special{pa 3508 1622}%
\special{pa 3476 1616}%
\special{pa 3446 1608}%
\special{pa 3416 1598}%
\special{pa 3386 1586}%
\special{pa 3356 1572}%
\special{pa 3328 1556}%
\special{pa 3302 1538}%
\special{pa 3278 1520}%
\special{pa 3254 1498}%
\special{pa 3230 1476}%
\special{pa 3210 1450}%
\special{pa 3190 1426}%
\special{pa 3172 1400}%
\special{pa 3154 1372}%
\special{pa 3140 1344}%
\special{pa 3126 1316}%
\special{pa 3112 1286}%
\special{pa 3102 1256}%
\special{pa 3092 1226}%
\special{pa 3084 1194}%
\special{pa 3076 1164}%
\special{pa 3070 1132}%
\special{pa 3068 1112}%
\special{sp}%
%
\special{pn 20}%
\special{sh 1}%
\special{ar 3070 1118 10 10 0  6.28318530717959E+0000}%
%
\special{pn 20}%
\special{sh 1}%
\special{ar 3996 1440 10 10 0  6.28318530717959E+0000}%
%
\special{pn 8}%
\special{ar 3420 1600 500 500  4.2619158 5.0467971}%
%
\special{pn 8}%
\special{ar 3420 1600 700 700  4.2619158 5.0472432}%
%
\special{pn 8}%
\special{pa 3206 1156}%
\special{pa 3112 964}%
\special{fp}%
%
\special{pn 8}%
\special{pa 3584 1128}%
\special{pa 3650 940}%
\special{fp}%
%
\special{pn 8}%
\special{ar 3860 1192 100 100  0.0000000 6.2831853}%
%
\special{pn 8}%
\special{ar 3420 1600 500 500  2.2858696 3.7547376}%
%
\special{pn 8}%
\special{ar 3420 1600 700 700  2.2858696 3.7547376}%
%
\special{pn 8}%
\special{pa 3000 1316}%
\special{pa 2840 1198}%
\special{fp}%
%
\special{pn 8}%
\special{pa 3090 1972}%
\special{pa 2962 2130}%
\special{fp}%
%
\special{pn 8}%
\special{ar 3680 2136 200 200  0.0000000 6.2831853}%
%
\special{pn 8}%
\special{pa 3262 2180}%
\special{pa 3206 2082}%
\special{fp}%
\special{pa 3262 2180}%
\special{pa 3158 2214}%
\special{fp}%
%
\special{pn 8}%
\special{ar 1190 1600 500 500  3.8021091 4.9738513}%
%
\special{pn 8}%
\special{ar 1190 1600 700 700  3.8021091 4.9738513}%
%
\special{pn 8}%
\special{pa 1318 1122}%
\special{pa 1362 932}%
\special{fp}%
%
\special{pn 8}%
\special{pa 796 1292}%
\special{pa 648 1170}%
\special{fp}%
%
\special{pn 8}%
\special{ar 1588 1158 100 100  0.0000000 6.2831853}%
%
\special{pn 8}%
\special{ar 1190 1600 500 500  1.9692287 3.1576233}%
%
\special{pn 8}%
\special{ar 1190 1600 700 700  1.9692287 3.1587338}%
%
\special{pn 8}%
\special{pa 692 1592}%
\special{pa 490 1588}%
\special{fp}%
%
\special{pn 8}%
\special{pa 992 2056}%
\special{pa 918 2234}%
\special{fp}%
%
\special{pn 8}%
\special{ar 1582 2068 200 200  0.0000000 6.2831853}%
%
\special{pn 8}%
\special{pa 1220 2200}%
\special{pa 1136 2124}%
\special{fp}%
\special{pa 1220 2200}%
\special{pa 1132 2264}%
\special{fp}%
\put(18.4000,-14.1000){\makebox(0,0)[lb]{$p_1$}}%
\put(17.0000,-10.5000){\makebox(0,0)[lb]{$a$}}%
%
\special{pn 8}%
\special{ar 450 1260 540 540  5.8925783 6.0036894}%
\special{ar 450 1260 540 540  6.0703560 6.1814672}%
\special{ar 450 1260 540 540  6.2481338 6.3592449}%
\special{ar 450 1260 540 540  6.4259116 6.5370227}%
\special{ar 450 1260 540 540  6.6036894 6.7148005}%
\special{ar 450 1260 540 540  6.7814672 6.8925783}%
\special{ar 450 1260 540 540  6.9592449 7.0703560}%
\special{ar 450 1260 540 540  7.1370227 7.2481338}%
\special{ar 450 1260 540 540  7.3148005 7.4259116}%
\special{ar 450 1260 540 540  7.4925783 7.5380114}%
%
\special{pn 8}%
\special{ar 2870 1050 470 470  6.1982835 6.3259431}%
\special{ar 2870 1050 470 470  6.4025388 6.5301984}%
\special{ar 2870 1050 470 470  6.6067942 6.7344537}%
\special{ar 2870 1050 470 470  6.8110495 6.9387090}%
\special{ar 2870 1050 470 470  7.0153048 7.1429644}%
\special{ar 2870 1050 470 470  7.2195601 7.3472197}%
\special{ar 2870 1050 470 470  7.4238154 7.5514750}%
\special{ar 2870 1050 470 470  7.6280707 7.7557303}%
\special{ar 2870 1050 470 470  7.8323261 7.9407200}%
\put(18.4000,-21.9000){\makebox(0,0)[lb]{$b$}}%
\put(39.5000,-21.9000){\makebox(0,0)[lb]{$b$}}%
\put(39.6000,-10.8000){\makebox(0,0)[lb]{$a$}}%
\put(40.9000,-14.2000){\makebox(0,0)[lb]{$p_2$}}%
\put(28.5000,-11.0000){\makebox(0,0)[lb]{$\overline{p_2}$}}%
\put(4.1000,-13.9000){\makebox(0,0)[lb]{$\overline{p_1}$}}%
\put(8.1000,-8.6000){\makebox(0,0)[lb]{$D_1^{\prime}$}}%
\put(3.1000,-20.6000){\makebox(0,0)[lb]{$D_1^{\prime \prime}$}}%
\put(24.6000,-17.5000){\makebox(0,0)[lb]{$D_2^{\prime \prime}$}}%
\put(33.0000,-8.4000){\makebox(0,0)[lb]{$D_2^{\prime}$}}%
\put(4.5000,-4.7000){\makebox(0,0)[lb]{$\mathcal{D}(D_1,p_1,D(a,b),\delta)$}}%
\put(26.4000,-4.7000){\makebox(0,0)[lb]{$\mathcal{D}(D_2,p_2,D(a,b),\delta)$}}%
\end{picture}%
\end{center}
\end{figure}

Observe that both the diagrams
have a unique chord such that the number of vertices
in the interior of the minor arc determined by the ends of the chord
is $\ge 2m$. Namely, the chords $\{ p_i, \overline{p_i}\}$, $i=1,2$.
Moreover, the number of the vertices in the interior of the
arc $\overline{p_i}p_i$ is $\ge 2(2m+1)$, and
that of the arc $p_i\overline{p_i}$ is $\le 4m$.
These imply the diagrams $\mathcal{D}(D_i,p_i,D(a,b),\delta)$
are of maximal index, and by assumption these two diagrams 
must coincide when we forget the labels of chords, and
the isomorphism between the two diagrams must maps
$p_1$ to $p_2$ and $\overline{p_1}$ to $\overline{p_2}$.
We conclude $C(D_1,p_1)=\pm C(D_2,p_2)$, and this proves the lemma.
\end{proof}

\begin{lem}
\label{hodai2}
Let $D_1$ and $D_2$ be labeled chord diagrams of $m$ chords,
and let $p_1$ and $p_2$ be vertices of $D_1$ and $D_2$, respectively.
Suppose $\mathcal{D}(D_1,p_1,D(a,b),\delta)=
\pm \mathcal{D}(D_2,p_2,D(a,b),\overline{\delta})$ in $\mathcal{C}$.
Then one of the following two occurs: 1) $D_1=D_2=\Omega_m$ up
to sign, and $p_1$ corresponds to an odd (resp.\ even) vertex
and so does $p_2$ to an even (resp.\ odd) one, or
2) there exist $c,d\ge 1$ such that we have
$D_1=D_2=D(c,d)$ up to sign, and
$p_1,p_2$ correspond to $\delta,\overline{\delta}$,
respectively.
\end{lem}

\begin{proof}
The picture of the diagram $\mathcal{D}(D_2,p_2,D(a,b),\overline{\delta})$
is obtained from the right diagram in Figure 12 by exchanging the role
of $a$ and $b$. By the same reason as before
this diagram is of maximal index. If we forget the labels of chords,
the diagrams $\mathcal{D}(D_1,p_1,D(a,b),\delta)$ and
$\mathcal{D}(D_2,p_2,D(a,b),\overline{\delta})$ must be isomorphic
by a unique map which maps $p_1$ to $\overline{p}_2$ and
$\overline{p}_1$ to $p_2$. Then $D_1^{\prime}, D_1^{\prime \prime}, D_2^{\prime},
D_2^{\prime \prime}$ must only have isolated chords.
If $D_1^{\prime}$ or $D_1^{\prime \prime}$
are the empty diagrams, the first conclusion follows.
If $D_1^{\prime}$ and $D_1^{\prime \prime}$ are both non-empty,
the second conclusion follows.
\end{proof}

As a corollary of the above two lemmas, we have:

\begin{cor}
\label{kekkyoku}
Let $D$ and $D^{\prime}$ be labeled chord diagrams of $m$ chords,
$p$ and $p^{\prime}$ vertices of $D$ and $D^{\prime}$, respectively,
and let $d,d^{\prime}\in \{ \delta, \overline{\delta} \}$. Suppose
$\mathcal{D}(D,p,D(a,b),d)=
\pm \mathcal{D}(D^{\prime},p^{\prime},D(a,b),d^{\prime})$ in
$\mathcal{C}$. Then $D=\pm D^{\prime} \in \mathcal{C}_m$.
\end{cor}

Now we are able to determine the center of $\mathcal{C}$.

\begin{thm}
\label{Z(C)}
$$Z(\mathcal{C})=\prod_{m\ge 2} \mathbb{Q}\Omega_m.$$
\end{thm}

\begin{proof}
Since the Lie algebra $\mathcal{C}$ is graded, it suffices to
show that any homogeneous element of degree $m$ which lies in the
center $Z(\mathcal{C})$ is actually a multiple of $\Omega_m$.
Suppose $X\in \mathcal{C}_m \cap Z(\mathcal{C})$ and write
$X$ as
\begin{equation}
\label{repX}
X=x \Omega_m+\sum_{(c,d)} x_{(c,d)}D(c,d)+\sum_i x_iD_i,\ x,x_{(c,d)},x_i \in \mathbb{Q}
\end{equation}
where the second term is a sum taken over $\{ (c,d);1\le c<d, m=c+d+1\}$, and
the third term is a sum taken over labeled chord diagrams $D_i$ not equal to
$\pm \Omega_m$ and $\pm D(c,d)$. We may assume the index of any $D_i$
is even, and $D_i\neq \pm D_j$ if $i\neq j$.

As in Lemmas \ref{hodai1} and \ref{hodai2}, let $a=m$ and $b=2m+1$.
Then we have
$$0=[X,D(a,b)]=\sum_{(c,d)}x_{(c,d)}[D(c,d),D(a,b)]
+\sum_i x_i [D_i,D(a,b)].$$

We claim that the elements $[D(c,d),D(a,b)]$ and
$[D_i,D(a,b)]$ are linearly independent in $\mathcal{C}$.
Assuming this claim, we have $x_{(c,d)}=x_i=0$ for all $(c,d)$ and $i$.
Thus $X=x\Omega_m$, and this will complete the proof.

Now we prove the claim. First we look at $[D(c,d),D(a,b)]$.
For simplicity we denote $\mathcal{D}(D(c,d),\delta,D(a,b),\delta)
=\mathcal{D}(\delta,\delta)$, etc. We have
$\mathcal{D}(\delta,\delta)=D(b+c,a+d)=-D(a+d,b+c)
=-\mathcal{D}(\overline{\delta},\overline{\delta})\in \mathcal{C}$,
and similarly we have $\mathcal{D}(\delta,\overline{\delta})=
-\mathcal{D}(\overline{\delta},\delta)$. Combining this with
(\ref{[D,D(a,b)]}), we have
$$[D(c,d),D(a,b)]=
\sum_{p\neq \delta,\overline{\delta}}\mathcal{D}(D(c,d),p,D(a,b),\delta)
+\sum_{p\neq \delta,\overline{\delta}}
\mathcal{D}(D(c,d),p,D(a,b),\overline{\delta}).$$
By Lemmas \ref{hodai1} and \ref{hodai2} and the fact that
$D(c,d)$ is of maximal index, the $2(2m-2)$ diagrams
appearing in this sum are distinct to each other,
even if we forget the labels of chords.
Therefore $[D(c,d),D(a,b)]$ is expressed as the sum of
$2(2m-2)$ distinct labeled chord diagrams which are linearly
independent in $\mathcal{C}$.

Next we look at $[D_i,D(a,b)]$.
We denote by $\iota=\iota(D_i)$ the index of $D_i$.
We have
$$[D_i,D(a,b)]=\sum_p \mathcal{D}(D_i,p,D(a,b),\delta)+
\sum_p \mathcal{D}(D_i,p,D(a,b),\overline{\delta}).$$
Again by Lemmas \ref{hodai1} and \ref{hodai2} this sum
equals $2m/\iota$ times the sum of $2\iota(D_i)$
distinct labeled chord diagrams linearly independent in
$\mathcal{C}$.

Set $\Delta=\{ D(c,d) \}_{(c,d)}\cup \{ D_i \}_i$ and
for each $D\in \Delta$, let $T_D \subset \mathcal{C}$ be the set of the
diagrams appearing in $[D,D(a,b)]$ described as above.
What we have observed is that $[D,D(a,b)]$ is a non-zero multiple
of $\sum_{\mathcal{D}\in T_D}\mathcal{D}$. Moreover,
by Corollary \ref{kekkyoku}, if $D,D^{\prime}\in \Delta$,
$D\neq D^{\prime}$, then $T_D\cap (\pm T_{D^{\prime}})=\emptyset$.
This shows $[D,D(a,b)]$, $D\in \Delta$ are linearly independent
and proves the claim.
\end{proof}

This proof also shows that if $X\in \mathcal{C}_m$ satisfies
$[X,\mathcal{C}_{3m+2}]=0$, then $X$ is in the center of $\mathcal{C}$.
The following theorem could be a supporting evidence for
Conjecture \ref{conj-Sigma}.

\begin{thm}\label{stabilized}
Denote $m(g) := \left[\frac{\displaystyle g-1}{\displaystyle 4}\right]+1$ 
for $g \ge 1$. Then we have 
$$
Z(\agminus)+N(\T_{2m(g)}) = \bigoplus^\infty_{m=2}\mathbb{Q}N(\omega^m) 
+ N(\T_{2m(g)}) \subset N(\T_1) \subset \agminus.
$$
\end{thm}

\begin{proof}
Let $u\in Z(\agminus)$ be a homogeneous element of degree
$< 2m(g)$. From (\ref{incl-ag}), we have
$u\in Z((\mathfrak{a}_g)^{\mathfrak{sp}})$.
By Lemma \ref{a-isom} (2), there uniquely exists $X\in \mathcal{C}_m$,
where $m<m(g)$, such that $a(X)=u$.
Since $u$ is in the center, $a([X,\mathcal{C}_{3m+2}])=[u,a(\mathcal{C}_{3m+2})]=0$.
On the other hand, $[X,\mathcal{C}_{3m+2}]\in \mathcal{C}_{4m+1}$
and $4m+1 \le g$ since $m<m(g)$. By Lemma \ref{a-isom} (2) and the remark
after the proof of Theorem \ref{Z(C)}, we see that $X$ is in the
center of $\mathcal{C}$. Hence $u=a(X)$ is a multiple of $a(\Omega_m)=N(\omega^m)$.

The other inclusion is clear since the map $a$ is surjective.
\end{proof}

As a corollary, we obtain
\begin{cor}\label{approx}
For any $u \in Z(\Qpi(\Sigma_{g,1}))$, there exists a polynomial 
$f(\zeta) \in \mathbb{Q}[\zeta] \subset \mathbb{Q}\pi$ such that
$$
u \equiv \vert f(\zeta)\vert \pmod{\Qpi(2m(g))}.
$$
\end{cor}
\begin{proof} We have $N\theta(u) \in Z(\agminus)$ by (\ref{incl-Z}). 
From Theorem \ref{stabilized} there exists a polynomial $h(\omega) 
\in \mathbb{Q}[\omega]$ such that $N\theta(u) \equiv Nh(\omega)
\pmod{N(\T_{2m(g)})}$. Since $\theta$ is symplectic, we have 
$\theta(\zeta^n) = \sum^\infty_{k=0}(1/k!)n^k\omega^k$. 
From Vandermonde's determinant
$$
\det\left(\frac{1}{k!}j^k\right)_{1\leq j,k\leq 2m(g)-1} = 
\left(\prod^{2m(g)-1}_{k=1}k!\right)^{-1} \prod_{j_i<j_2}(j_2-j_1) \neq 0,
$$
there exists a polynomial $f(\zeta) \in \mathbb{Q}[\zeta]$ such that 
$\theta(f(\zeta)) \equiv h(\omega) \pmod{\T_{2m(g)}}$. 
Hence we have $N\theta(u) \equiv N\theta(f(\zeta))
\pmod{N(\T_{2m(g)})}$, and so 
$u \equiv \vert f(\zeta)\vert \pmod{\Qpi(2m(g))}$, as was to be shown.
\end{proof}

\section{Surface of infinite genus}\label{s:infinite}

In this section we prove Theorem \ref{main}.

\subsection{Inductive system of surfaces}\label{s:inductive}
As in the Introduction, we consider the embedding
$$
i^{g}_{g+1}: \Sigma_{g,1}\to \Sigma_{g+1,1}
$$
given by gluing the surface $\Sigma_{1,2}$ to the surface $\Sigma_{g,1}$ 
along the boundary. These embeddings constitute an inductive system
of oriented surfaces $\{\Sigma_{g,1}, i^h_g\}_{h \leq g}$. Here 
$i^h_g: \Sigma_{h,1}\to\Sigma_{g,1}$ is the composite of 
the embeddings $i^{h}_{h+1}, i^{h+1}_{h+2},\cdots, i^{g-1}_g$. 
Choose a basepoint $*_g$ on the boundary $\partial\Sigma_{g,1}$.
For the rest of the paper, we often write simply
$$
\pi^{(g)} = \pi_1(\Sigma_{g,1}, *_g), \quad
\hat{\pi}^{(g)} = \hat{\pi}(\Sigma_{g,1}), \quad
H^{(g)} = H_1(\Sigma_{g,1}; \mathbb{Q}) \quad\mbox{and}\quad
\T^{(g)} = \prod^\infty_{m=1}(H^{(g)})^{\otimes m}.
$$
\begin{lem}\label{inj-GLA}
The inclusion map $i^h_g$ induces an injective map of homotopy sets
$i^h_g: \hat{\pi}^{(h)} \to \hat{\pi}^{(g)}$. In particular, the map
$$
i^h_g: \Qpi^{(h)} \to \Qpi^{(g)}
$$ 
on the Goldman Lie algebras is an injective homomorphism 
of Lie algebras.
\end{lem}
\begin{proof}
Choose a simple path $\ell: [0,1] \to \Sigma_{g,1} \setminus 
\overset{\circ}{\Sigma_{h,1}}$ connecting the basepoint $*_g$ to $*_h$. 
Here we denote by $\overset{\circ}{\Sigma_{h,1}}$ the interior 
of the surface $\Sigma_{h,1}$. 
The map $i^g_h: \pi^{(h)} \to \pi^{(g)}$ given by $x \mapsto \ell x \ell^{-1}$ 
is an injective homomorphism which induces the map $i^h_g: \hat{\pi}^{(h)} 
\to \hat{\pi}^{(g)}$. There exists a group homomorphism
$r^g_h: \pi^{(g)} \to \pi^{(h)}$ satisfying $r^g_h\circ i^h_g = 1_{\pi^{(h)}}$. 
In fact, if $\{x_1, \dots, x_{2h}\} \subset \pi^{(h)}$ is a free generating system 
of  $\pi^{(h)}$, we may choose $x_i \in \pi^{(g)}$ for $2h+1 \leq i \leq 2g$ 
such that $\{\ell x_1\ell^{-1}, \dots, \ell x_{2h}\ell^{-1}, 
x_{2h+1}, \dots, x_{2g}\}$ is a free generating system of $\pi^{(g)}$. 
If we define $r^g_h$ by 
$r^g_h(\ell x_i\ell^{-1}) = x_i$ for $1 \leq i \leq 2h$, and 
$r^g_h(x_j) = 1$ for $2h+1\leq j\leq 2g$, then we have  
$r^g_h\circ i^h_g = 1_{\pi^{(h)}}$. \par
Let  $x$ and $y$ be elements in $\pi^{(h)}$. 
Suppose $i^h_g(x)$ is conjugate to $i^h_g(y)$. Then there exists an 
element $z \in \pi^{(g)}$ such that $i^h_g(y) = z i^h_g(x)z^{-1}$. 
Applying the homomorphism $r^g_h$, we obtain 
$y=r^g_h(z) x r^g_h(z)^{-1}$. Hence $x$ is conjugate to $y$. 
This proves the first half of the lemma. \par
From the first half, the map $i^h_g: \Qpi^{(h)} \to \Qpi^{(g)}$ is 
injective. It is a homomorphism of Lie algebras by the definition of 
the Goldman bracket.
\end{proof}

Recall from \S2.3 the decreasing filtration $\mathbb{Q}\hat{\pi}(p)$.

\begin{lem}\label{filter} For any $p \geq 1$ and $h\leq g$, we have 
$$
\Qpi^{(h)}(p) = (i^h_g)^{-1}(\Qpi^{(g)}(p))
$$
\end{lem}
\begin{proof} Choose a Magnus expansion $\theta': \pi^{(h)} \to 
\T^{(h)}$ and extend it to a Magnus expansion $\theta'': 
\pi^{(g)} \to \T^{(g)}$. We have a commutative diagram
$$
\begin{CD}
\mathbb{Q}\pi^{(h)} @>{\theta'}>> \T^{(h)}\\
@V{i^h_g}VV @V{i^h_g}VV\\
\mathbb{Q}\pi^{(g)} @>{\theta''}>> \T^{(g)}.
\end{CD}
$$
Here the right $i^h_g$ is induced by the inclusion homomorphism
${i^h_g}_*: H^{(h)} = H_1(\Sigma_{h,1};\mathbb{Q}) \to 
H^{(g)} = H_1(\Sigma_{g,1};\mathbb{Q})$. 
Using the map $r^g_h$ introduced in the proof of Lemma \ref{inj-GLA}, 
we obtain $(i^h_g)^{-1}(\T^{(g)}_p) = \T^{(h)}_p$ and 
$(i^h_g)^{-1}([\T^{(g)}, \T^{(g)}]) = [\T^{(h)}, \T^{(h)}]$. 
Hence, for $u \in \Qpi^{(h)}$, the condition 
$\theta''(i^h_g(u)) - \varepsilon(u) \in \T^{(g)}_p+[\T^{(g)}, \T^{(g)}]$ 
is equivalent to 
$\theta'(u) - \varepsilon(u) \in \T^{(h)}_p+[\T^{(h)}, \T^{(h)}]$. 
From Lemma \ref{characterization}, these conditions are equivalent to 
$i^h_g\vert u\vert \in \Qpi^{(g)}(p)$ and $\vert u\vert \in \Qpi^{(h)}(p)$, 
respectively. This proves the lemma.
\end{proof}

We denote by $\Sigma_{\infty,1}$ the inductive limit 
of the system $\{\Sigma_{g,1}, i^h_g\}_{h \leq g}$
$$
\Sigma_{\infty,1} := \varinjlim_{g\to\infty}\Sigma_{g,1}.
$$
This is an oriented connected paracompact surface. 
We regard $\Sigma_{g,1}$ as a subsurface of $\Sigma_{\infty,1}$ and 
denote the inclusion map by $i^g_\infty: \Sigma_{g,1}\to 
\Sigma_{\infty,1}$. 
For any compact subset $K \subset \Sigma_{\infty,1}$, 
there exists a sufficiently large $g$ such that $K \subset \Sigma_{g,1}$. 
In particular, the Goldman Lie algebra $\Qpi(\Sigma_{\infty,1})$ is 
exactly the inductive limit of the Lie algebras $\Qpi(\Sigma_{g,1})$'s
\begin{equation}
\Qpi(\Sigma_{\infty,1}) = \varinjlim_{g\to\infty}\Qpi(\Sigma_{g,1}).
\label{limit}
\end{equation}
From Lemma \ref{inj-GLA}, the inclusion homomorphism
\begin{equation}
i^g_\infty: \Qpi(\Sigma_{g,1}) \to \Qpi(\Sigma_{\infty,1})
\label{inj-infty}
\end{equation}
is injective. \par

\subsection{Proof of Theorem \ref{main}}\label{s:proof}

In this subsection we prove Theorem \ref{main}. 
It is clear $\mathbb{Q}1 \subset Z(\Qpi(\Sigma_{\infty,1}))$. 
We assume there exists an element $u \in Z(\Qpi(\Sigma_{\infty,1}))
\setminus \mathbb{Q}1$, and deduce a contradiction. 
By (\ref{limit}), we have $u \in \Qpi(\Sigma_{g_0,1})$ 
for some $g_0 \geq 1$.
From (\ref{intersection}) and the assumption $u \not\in \mathbb{Q}1$, 
there exists some $p \geq 1$ such that $u \not\in \Qpi^{(g_0)}(p)$. 
We choose the minimum $p$ satisfying this property. 
By Lemma \ref{filter}, we have 
\begin{equation}
i^{g_0}_g(u) \not\in \Qpi^{(g)}(p)
\label{notin}
\end{equation}
for any $g \geq g_0$.\par
There exists some $g_1\geq g_0$ 
such that $2m(g) \geq p$ for any $g \geq g_1$. 
Denote $h := g_1$ and $g:= h+1$. 
Choose a non-null homologous based loop 
$\alpha \in \pi^{(g)} = \pi_1(\Sigma_{g,1}, *_g)$ inside 
the subsurface $\Sigma_{1,2} \subset \Sigma_{g,1}$.
We denote the boundary loops of $\Sigma_{h,1}$ and $\Sigma_{g,1}$ 
by $\gamma$ and $\zeta$, respectively. The loops $\gamma$ and 
$\alpha$ are disjoint. See Figure 13.

\begin{figure}
\begin{center}
\caption{$\Sigma_{h,1}$ and $\Sigma_{g,1}$}

\vspace{0.2cm}

\unitlength 0.1in
\begin{picture}( 38.0000, 12.0000)(  1.0000,-16.0000)
%
\special{pn 13}%
\special{ar 3800 1000 100 600  4.7123890 6.2831853}%
\special{ar 3800 1000 100 600  0.0000000 1.5707963}%
%
\special{pn 13}%
\special{ar 3800 1000 100 600  1.5707963 4.7123890}%
%
\special{pn 13}%
\special{ar 2900 1000 200 200  0.0000000 6.2831853}%
%
\special{pn 8}%
\special{ar 2300 1000 100 600  4.7123890 4.8838176}%
\special{ar 2300 1000 100 600  4.9866747 5.1581033}%
\special{ar 2300 1000 100 600  5.2609604 5.4323890}%
\special{ar 2300 1000 100 600  5.5352461 5.7066747}%
\special{ar 2300 1000 100 600  5.8095318 5.9809604}%
\special{ar 2300 1000 100 600  6.0838176 6.2552461}%
\special{ar 2300 1000 100 600  6.3581033 6.5295318}%
\special{ar 2300 1000 100 600  6.6323890 6.8038176}%
\special{ar 2300 1000 100 600  6.9066747 7.0781033}%
\special{ar 2300 1000 100 600  7.1809604 7.3523890}%
\special{ar 2300 1000 100 600  7.4552461 7.6266747}%
\special{ar 2300 1000 100 600  7.7295318 7.8539816}%
%
\special{pn 8}%
\special{ar 2300 1000 100 600  1.5707963 4.7123890}%
%
\special{pn 8}%
\special{ar 3500 1000 100 600  4.7123890 4.8838176}%
\special{ar 3500 1000 100 600  4.9866747 5.1581033}%
\special{ar 3500 1000 100 600  5.2609604 5.4323890}%
\special{ar 3500 1000 100 600  5.5352461 5.7066747}%
\special{ar 3500 1000 100 600  5.8095318 5.9809604}%
\special{ar 3500 1000 100 600  6.0838176 6.2552461}%
\special{ar 3500 1000 100 600  6.3581033 6.5295318}%
\special{ar 3500 1000 100 600  6.6323890 6.8038176}%
\special{ar 3500 1000 100 600  6.9066747 7.0781033}%
\special{ar 3500 1000 100 600  7.1809604 7.3523890}%
\special{ar 3500 1000 100 600  7.4552461 7.6266747}%
\special{ar 3500 1000 100 600  7.7295318 7.8539816}%
%
\special{pn 8}%
\special{ar 3500 1000 100 600  1.5707963 4.7123890}%
%
\special{pn 8}%
\special{ar 2900 1000 400 400  1.5707963 6.2831853}%
%
\special{pn 8}%
\special{ar 3700 1000 400 400  1.4959365 3.1415927}%
%
\special{pn 8}%
\special{pa 2900 1400}%
\special{pa 3700 1400}%
\special{fp}%
%
\special{pn 13}%
\special{ar 1700 1000 200 200  0.0000000 6.2831853}%
%
\special{pn 13}%
\special{ar 700 1000 200 200  0.0000000 6.2831853}%
\put(11.0000,-10.5000){\makebox(0,0)[lb]{$\cdots$}}%
%
\special{pn 13}%
\special{ar 700 1000 600 600  1.5707963 4.7123890}%
%
\special{pn 13}%
\special{pa 700 400}%
\special{pa 3800 400}%
\special{fp}%
%
\special{pn 13}%
\special{pa 700 1600}%
\special{pa 3800 1600}%
\special{fp}%
%
\special{pn 8}%
\special{pa 2900 600}%
\special{pa 2810 540}%
\special{fp}%
\special{pa 2900 600}%
\special{pa 2820 680}%
\special{fp}%
%
\special{pn 8}%
\special{pa 2200 1000}%
\special{pa 2140 910}%
\special{fp}%
\special{pa 2200 1000}%
\special{pa 2260 910}%
\special{fp}%
%
\special{pn 8}%
\special{pa 3400 1000}%
\special{pa 3340 910}%
\special{fp}%
\special{pa 3400 1000}%
\special{pa 3460 910}%
\special{fp}%
%
\special{pn 20}%
\special{sh 1}%
\special{ar 3730 1400 10 10 0  6.28318530717959E+0000}%
\put(20.8000,-6.0000){\makebox(0,0)[lb]{$\gamma$}}%
\put(33.0000,-6.0000){\makebox(0,0)[lb]{$\zeta$}}%
\put(28.7000,-15.2000){\makebox(0,0)[lb]{$\alpha$}}%
\end{picture}%
\end{center}
\end{figure}
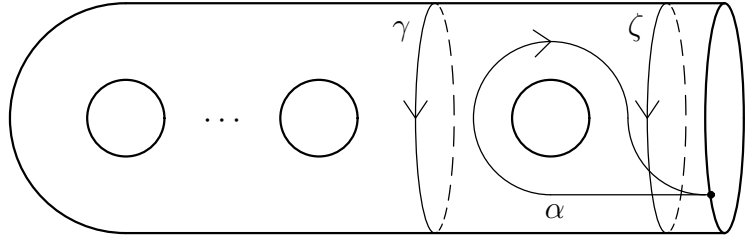

From (\ref{inj-infty}) and Lemma \ref{inj-GLA}, the homomorphisms 
$$
\Qpi(\Sigma_{h,1}) \overset{i^h_g}\to
\Qpi(\Sigma_{g,1}) \overset{i^g_\infty}\to
\Qpi(\Sigma_{\infty,1}) 
$$ 
are injective. Hence we may regard $u \in Z(\Qpi(\Sigma_{h,1}))
\cap Z(\Qpi(\Sigma_{g,1}))$. 
By Corollay \ref{approx} we have polynomials 
$f_h(\gamma) \in \mathbb{Q}[\gamma]$ and 
$f_h(\zeta) \in \mathbb{Q}[\zeta]$ such that 
$$
u \equiv \vert f_h(\gamma)\vert \pmod{\Qpi^{(h)}(p)}, 
\quad\mbox{and}\quad
u \equiv \vert f_g(\zeta)\vert \pmod{\Qpi^{(g)}(p)}. 
$$
By Lemma \ref{filter} we have 
$$
\vert f_h(\gamma)\vert\equiv u \equiv 
\vert f_g(\zeta)\vert \pmod{\Qpi^{(g)}(p)}. 
$$
\par
Choose a symplectic expansion $\theta: \pi^{(g)} \to \T^{(g)}$. 
For the rest of the proof, we drop the suffix $\empty^{(g)}$. 
If $v \in \Qpi(p)$, then $N\theta(v) \in N(\T_p)$ and so 
$(N\theta(v))\theta(\alpha) = (N\theta(v))\theta(\alpha-1)
\in \T_{p+1-2} = \T_{p-1}$, since $\theta(\alpha-1) \in \T_1$. 
Hence we have 
$(N\theta(f_h(\gamma)))\theta(\alpha) \equiv
(N\theta(f_g(\zeta)))\theta(\alpha) \pmod{\T_{p-1}}$. 
Moreover we have $(N\theta(f_h(\gamma)))\theta(\alpha) = 0$ 
by \cite{KK} Theorem 1.2.2, 
since the free loop $\gamma$ and the based loop $\alpha$ are 
disjoint. Thus we obtain
\begin{equation}
(N\theta(f_g(\zeta)))\theta(\alpha) \in \T_{p-1}.
\label{vanish}
\end{equation}
On the other hand, we have $\vert f_g(\zeta)\vert \not\in \Qpi(p)$ 
because $u \not\in \Qpi(p)$. $\theta(f_g(\zeta))$ is a power series
in the symplectic form $\omega$. Hence, since $p$ is the minimum, 
$p$ is odd $\geq 5$, and $N(\theta(f_g(\zeta))) =cN(\omega^{(p-1)/2})
+({\rm higher\ term})$ for some non-zero constant $c \in \mathbb{Q}$
(we have $p\neq 3$ since $N(\omega)=0$). Then we have 
\begin{eqnarray*}
(N\theta(f_g(\zeta)))\theta(\alpha) &\equiv&
cN(\omega^{(p-1)/2})([\alpha])\\ &=&
((p-1)/2)c(-[\alpha]\omega^{(p-3)/2} + \omega^{(p-3)/2}[\alpha])
\not\equiv 0 \pmod{\T_{p-1}}.
\end{eqnarray*}
Here the equality between the second and the third terms follows from
the computation that if $\{A_i, B_i\}^g_{i=1} \subset H$ is a symplectic basis
and $m\ge 2$, then
\begin{eqnarray*}
N(\omega^{m})([\alpha]) &=&
m\sum_{i=1}^g \left(
\begin{array}{c}
([\alpha]\cdot A_i)B_i\omega^{m-1}- ([\alpha]\cdot B_i)A_i\omega^{m-1} \\
+ ([\alpha]\cdot B_i)\omega^{m-1}A_i- ([\alpha]\cdot A_i)\omega^{m-1}B_i
\end{array} \right) \\
&=& -m[\alpha]\omega^{m-1} + m\omega^{m-1}[\alpha].
\end{eqnarray*}
Namely we have $(N\theta(f_g(\zeta)))\theta(\alpha) \not\in\T_{p-1}$.
This contradicts (\ref{vanish}). 
Hence we obtain $Z(\Qpi(\Sigma_{\infty,1})) \subset \mathbb{Q}1$. 
This completes the proof of Theorem \ref{main}.

\section{Appendix: The Lie algebra of linear chord diagrams}

The Lie bracket on the space $\mathcal{C}$ of oriented chord diagrams is
extended to a bracket on the space of linear chord diagrams
$$
[\,,\,]: \mathcal{LC}_m\otimes\mathcal{LC}_{m^{\prime}}
\to \mathcal{LC}_{m+m^{\prime}-1},
$$
which makes the direct sum
$$
\mathcal{LC} := \bigoplus^\infty_{m=1}\mathcal{LC}_m
$$
a Lie algebra.
We define the bracket by using the stable isomorphism $a: \mathcal{LC}
\to \Der(T)^\fsp$ in Lemma 3.1.1 and the Lie algebra structure
on $\Der(T)$, the derivation algebra of $T$. Here $T :=
\bigoplus^\infty_{m=0} H^{\otimes m}$ is the tensor algebra of $H$, 
the rational symplectic vector space of genus $g \geq 1$. As before, we
identify the dual
$H^*=\Hom(H, \mathbb{Q})$ with $H$ by the Poincar\'e duality $H \cong
H^*$, $X \mapsto (Y\mapsto Y\cdot X)$. Then the restriction map to the 
subspace $H$ identifies the space $\Der(T)$ with the space $\Hom(H, T) =
H^*\otimes T = H\otimes T=\bigoplus^{\infty}_{m=1} H^{\otimes m}$. \par
It should be remarked the set of linear chord diagrams of $m$ chords 
with the standard label is a basis of the space $\mathcal{LC}_m$. 
Here $C = \{(i_1, j_1), \dots, (i_m, j_m)\}$ is a linear chord diagram
of $m$ chords  with the standard label, if and only if $\{i_1\dots, i_m,
j_1, \dots, j_m\} = \{1, 2, \dots, 2m\}$ and $i_k < j_k$ for any $k$
(see the proofs of Lemmas 3.1.3 and 3.1.4). 
For the rest of this appendix, we regard $\mathcal{LC}$ as the vector
space spanned by the (unlabeled) linear chord diagrams. Thus we identify 
the labeled linear chord diagram $C$
with the fixed-point free involution
$\sigma(C) :=  (i_1, j_1)\cdots (i_m, j_m) \in \mathfrak{S}_{2m}$. 
The invariant tensor $a(C) \in (H^{\otimes 2m})^\fsp$ is defined as in
\S3.1 and the map $a: \mathcal{LC}_m \to (H^{\otimes 2m})^\fsp$ is 
a stable isomorphism (Lemma 3.1.1). This stable isomorphism 
induces a Lie algebra structure on the space $\mathcal{LC}$ such that 
$a: \mathcal{LC} \to \Der(T)^\fsp$ is a Lie algebra homomorphism. \par
In order to describe the bracket on $\mathcal{LC}$, 
we introduce new amalgamations of two linear chord diagrams. 
Let $C$ and $C'$ be linear chord diagrams of $m$ and $l$ chords, 
respectively. They are regarded as involutions $\sigma = \sigma(C) 
\in \mathfrak{S}_{2m}$ and $\sigma' = \sigma(C') 
\in \mathfrak{S}_{2l}$. For $2 \leq t \leq 2l$, we define 
the $t$-th amalgamation $C\ast_{t}C'$ as an involution 
$\sigma'' = \sigma(C\ast_{t}C') \in \mathfrak{S}_{2m+2l-2}$ by
\begin{eqnarray*}
&&\sigma''(\sigma(1)+t-2) := f_{m,t}(\sigma'(t))\\
&&\sigma''(f_{m,t}(\sigma'(t))):=\sigma(1)+t-2 \\
&&\sigma''(k) := \begin{cases}
f_{m,t}(\sigma'(k)), & 
\mbox{if $k \leq t-1$ and $k\neq\sigma'(t)$,}\\
\sigma(k-t+2) +t-2, &
\mbox{if $t\leq k \leq t+2m-2$ and $k\neq\sigma(1)+t-2$,}\\
f_{m,t}(\sigma'(k-2m+2)), & 
\mbox{if $t+2m-1\leq k$ and $k-2m+2\neq\sigma'(t)$}.
\end{cases}
\end{eqnarray*}
Here $f_{m,t}: \{1,\dots,t-1,t+1,\dots,2l\} \to 
\{1,2,\dots,2m+2l-2\}$ is defined by 
$$
f_{m,t}(k) := 
\begin{cases}
k, & \mbox{if $k \leq t-1$,}\\
k+2m-2, & \mbox{if $k \geq t+1$.}
\end{cases}
$$
In other words, we delete the $t$-th vertex from $C'$ and 
the first vertex from $C$, insert the deleted $C$ into the 
$t$-th hole of the deleted $C'$, and connect the vertices 
$\sigma(C)(1)$ and $\sigma(C')(t)$. The resulting linear chord 
diagram with the standard label is exactly the $t$-th amalgamation 
$C\ast_t C' \in \mathcal{LC}_{m+l-1}$. See Figure 14. Interchanging the role
of $C$ and $C^{\prime}$, we can define the $s$-th amalgamation 
$C'\ast_s C$ for $2 \leq s \leq 2m$.

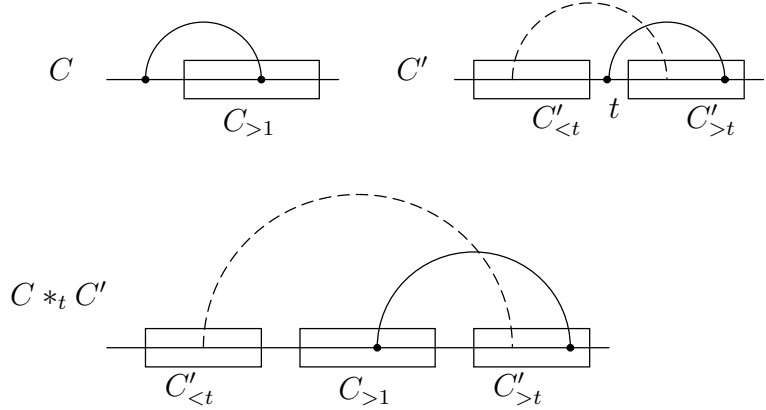
\begin{figure}
\begin{center}
\caption{the $t$-th amalgamation $C\ast_t C^{\prime}$}

\vspace{0.2cm}

\unitlength 0.1in
\begin{picture}( 39.0000, 19.3000)(  3.0000,-25.3000)
%
\special{pn 8}%
\special{pa 800 1000}%
\special{pa 2000 1000}%
\special{fp}%
%
\special{pn 20}%
\special{sh 1}%
\special{ar 1000 1000 10 10 0  6.28318530717959E+0000}%
%
\special{pn 8}%
\special{ar 1300 1000 300 300  3.1415927 6.2831853}%
%
\special{pn 8}%
\special{pa 1200 900}%
\special{pa 1900 900}%
\special{pa 1900 1100}%
\special{pa 1200 1100}%
\special{pa 1200 900}%
\special{fp}%
%
\special{pn 8}%
\special{pa 2600 1000}%
\special{pa 4200 1000}%
\special{fp}%
%
\special{pn 8}%
\special{pa 2700 900}%
\special{pa 3300 900}%
\special{pa 3300 1100}%
\special{pa 2700 1100}%
\special{pa 2700 900}%
\special{fp}%
%
\special{pn 8}%
\special{pa 3500 1100}%
\special{pa 4100 1100}%
\special{pa 4100 900}%
\special{pa 3500 900}%
\special{pa 3500 1100}%
\special{fp}%
%
\special{pn 20}%
\special{sh 1}%
\special{ar 3390 1000 10 10 0  6.28318530717959E+0000}%
%
\special{pn 8}%
\special{ar 3700 1000 300 300  3.1415927 6.2831853}%
%
\special{pn 8}%
\special{ar 3300 1000 400 400  3.1415927 3.2915927}%
\special{ar 3300 1000 400 400  3.3815927 3.5315927}%
\special{ar 3300 1000 400 400  3.6215927 3.7715927}%
\special{ar 3300 1000 400 400  3.8615927 4.0115927}%
\special{ar 3300 1000 400 400  4.1015927 4.2515927}%
\special{ar 3300 1000 400 400  4.3415927 4.4915927}%
\special{ar 3300 1000 400 400  4.5815927 4.7315927}%
\special{ar 3300 1000 400 400  4.8215927 4.9715927}%
\special{ar 3300 1000 400 400  5.0615927 5.2115927}%
\special{ar 3300 1000 400 400  5.3015927 5.4515927}%
\special{ar 3300 1000 400 400  5.5415927 5.6915927}%
\special{ar 3300 1000 400 400  5.7815927 5.9315927}%
\special{ar 3300 1000 400 400  6.0215927 6.1715927}%
\special{ar 3300 1000 400 400  6.2615927 6.2831853}%
\put(5.0000,-10.0000){\makebox(0,0)[lb]{$C$}}%
\put(23.0000,-10.0000){\makebox(0,0)[lb]{$C^{\prime}$}}%
\put(34.0000,-12.0000){\makebox(0,0)[lb]{$t$}}%
%
\special{pn 8}%
\special{pa 1000 2500}%
\special{pa 1600 2500}%
\special{pa 1600 2300}%
\special{pa 1000 2300}%
\special{pa 1000 2500}%
\special{fp}%
%
\special{pn 8}%
\special{pa 1800 2500}%
\special{pa 2500 2500}%
\special{pa 2500 2300}%
\special{pa 1800 2300}%
\special{pa 1800 2500}%
\special{fp}%
%
\special{pn 8}%
\special{pa 2700 2300}%
\special{pa 3300 2300}%
\special{pa 3300 2500}%
\special{pa 2700 2500}%
\special{pa 2700 2300}%
\special{fp}%
%
\special{pn 8}%
\special{ar 2700 2400 500 500  3.1415927 6.2831853}%
%
\special{pn 8}%
\special{ar 2100 2400 800 800  3.1415927 3.2165927}%
\special{ar 2100 2400 800 800  3.2615927 3.3365927}%
\special{ar 2100 2400 800 800  3.3815927 3.4565927}%
\special{ar 2100 2400 800 800  3.5015927 3.5765927}%
\special{ar 2100 2400 800 800  3.6215927 3.6965927}%
\special{ar 2100 2400 800 800  3.7415927 3.8165927}%
\special{ar 2100 2400 800 800  3.8615927 3.9365927}%
\special{ar 2100 2400 800 800  3.9815927 4.0565927}%
\special{ar 2100 2400 800 800  4.1015927 4.1765927}%
\special{ar 2100 2400 800 800  4.2215927 4.2965927}%
\special{ar 2100 2400 800 800  4.3415927 4.4165927}%
\special{ar 2100 2400 800 800  4.4615927 4.5365927}%
\special{ar 2100 2400 800 800  4.5815927 4.6565927}%
\special{ar 2100 2400 800 800  4.7015927 4.7765927}%
\special{ar 2100 2400 800 800  4.8215927 4.8965927}%
\special{ar 2100 2400 800 800  4.9415927 5.0165927}%
\special{ar 2100 2400 800 800  5.0615927 5.1365927}%
\special{ar 2100 2400 800 800  5.1815927 5.2565927}%
\special{ar 2100 2400 800 800  5.3015927 5.3765927}%
\special{ar 2100 2400 800 800  5.4215927 5.4965927}%
\special{ar 2100 2400 800 800  5.5415927 5.6165927}%
\special{ar 2100 2400 800 800  5.6615927 5.7365927}%
\special{ar 2100 2400 800 800  5.7815927 5.8565927}%
\special{ar 2100 2400 800 800  5.9015927 5.9765927}%
\special{ar 2100 2400 800 800  6.0215927 6.0965927}%
\special{ar 2100 2400 800 800  6.1415927 6.2165927}%
\special{ar 2100 2400 800 800  6.2615927 6.2831853}%
\put(14.0000,-13.0000){\makebox(0,0)[lb]{$C_{>1}$}}%
\put(30.0000,-13.0000){\makebox(0,0)[lb]{$C^{\prime}_{<t}$}}%
\put(38.0000,-13.0000){\makebox(0,0)[lb]{$C^{\prime}_{>t}$}}%
%
\special{pn 8}%
\special{pa 800 2400}%
\special{pa 3400 2400}%
\special{fp}%
\put(3.0000,-22.0000){\makebox(0,0)[lb]{$C\ast_t C^{\prime}$}}%
\put(20.0000,-27.0000){\makebox(0,0)[lb]{$C_{>1}$}}%
\put(11.0000,-27.0000){\makebox(0,0)[lb]{$C^{\prime}_{<t}$}}%
\put(28.0000,-27.0000){\makebox(0,0)[lb]{$C^{\prime}_{>t}$}}%
%
\special{pn 20}%
\special{sh 1}%
\special{ar 1600 1000 10 10 0  6.28318530717959E+0000}%
%
\special{pn 20}%
\special{sh 1}%
\special{ar 4000 1000 10 10 0  6.28318530717959E+0000}%
%
\special{pn 20}%
\special{sh 1}%
\special{ar 2200 2400 10 10 0  6.28318530717959E+0000}%
%
\special{pn 20}%
\special{sh 1}%
\special{ar 3200 2400 10 10 0  6.28318530717959E+0000}%
\end{picture}%
\end{center}
\end{figure}

By a straightforward computation we see that the bracket 
on the space $\Der(T) = \bigoplus^{\infty}_{m=1}H^{\otimes m}$ is given by 
\begin{eqnarray*}
[X_1\cdots X_p, Y_1\cdots Y_q] 
&=& \sum^q_{t=2}(Y_t\cdot X_1)Y_1Y_2\cdots Y_{t-1}X_2\cdots
X_pY_{t+1}\cdots Y_q \\
&&-\sum^p_{s=1}(X_s\cdot Y_1)X_1X_2\cdots X_{s-1}Y_2\cdots
Y_qX_{s+1}\cdots X_p
\end{eqnarray*}
for $X_s$, $Y_t \in H$. Hence, by a similar argument 
to \S3.1, we have 
\begin{equation}
[C, C'] = -\sum^{2l}_{t=2}C\ast_t C' + 
\sum^{2m}_{s=2}C'\ast_s C.
\label{LCbracket}
\end{equation}
It is easy to compute the center and the homology of the Lie algebra 
$\mathcal{LC}$. We denote $E_0 := -\frac12\{\{1,2\}\} \in
\mathcal{LC}_1$. Then we have $(-2E_0)\ast_t C= C\ast_2(-2E_0) = C$ 
for any $t$. Hence $\mathcal{LC}_m$ is just the eigenspace of 
the operator ${\rm ad}(E_0)$ corresponding to the eigenvalue $m-1 (\geq 0)$.
This observation implies the center of $\mathcal{LC}$ vanishes
\begin{equation}
Z(\mathcal{LC}) = 0.
\end{equation}
Using the Lie derivative $\mathcal{L}_{E_0}$, we can prove that the standard 
chain complex $C_*(\mathcal{LC})$ is quasi-isomorphic to the
$E_0$-invariant subcomplex $C_*(\mathcal{LC})^{E_0} =
C_*(\mathcal{LC}_1)$. Thus we obtain
\begin{equation}
H_*(\mathcal{LC}) =
\begin{cases}
\mathbb{Q}, & \mbox{if $*=0,1$,}\\
0, & \mbox{otherwise.}
\end{cases} 
\end{equation}
\par
We denote by $W_1 := \mathbb{Q}[x]\frac{d}{dx}$ the Lie algebra of 
polynomial vector fields in one variable $x$. The subalgebras 
$L_0 := x\mathbb{Q}[x]\frac{d}{dx}$ and 
$L_1 := x^2\mathbb{Q}[x]\frac{d}{dx}$ play important roles 
in Gel'fand-Fuks theory (cf., e.g., \cite{Fuks}). 
The formula (\ref{LCbracket}) implies immediately that the surjection 
$$
\kappa: \mathcal{LC} \to L_0
$$
assigning $-2x^m\frac{d}{dx}$ to each linear chord diagram of $m$ chords
is a Lie algebra homomorphism. The vector field $\kappa(E_0) =
x\frac{d}{dx}$ is just the Euler operator. \par
By analogy with the Lie subalgebra $L_1$, we consider the Lie algebra 
$\mathcal{LC}^1 := \bigoplus^\infty_{m=2}\mathcal{LC}_m$. 
The homology group $H_*(\mathcal{LC}^1)$ is decomposed into 
the eigenspaces of the action of $E_0$. We denote by
$H_*(\mathcal{LC}^1)_{(k)}$ the eigenspace corresponding to the eigenvalue 
$k \geq 1$. The first homology group 
$H_1(\mathcal{LC}^1)_{(k)}$ does not vanish for any integer $k \geq 1$, 
and its dimension diverges when $k$ goes to the infinity. 
The proof will appear elsewhere. The generating function of the Euler 
characteristics $\sum^\infty_{k=1}\chi(H_*(\mathcal{LC}^1)_{(k)})x^k$ 
can be computed as
$$
-3x-12x^2-61x^3-570x^4-6600x^5-91910x^6
-1460655x^7-26064990x^8-\cdots.
$$
This is completely different from the homology of the Lie subalgebra 
$L_1$ given by Goncharova \cite{Gon}.

\noindent \textsc{Nariya Kawazumi\\
Department of Mathematical Sciences,\\
University of Tokyo,\\
3-8-1 Komaba Meguro-ku Tokyo 153-8914 JAPAN}\\
\noindent \texttt{E-mail address: kawazumi@ms.u-tokyo.ac.jp}

\vspace{0.5cm}

\noindent \textsc{Yusuke Kuno\\
Department of Mathematics,\\
Tsuda College,\\
2-1-1, Tsuda-Machi, Kodaira-shi, Tokyo 187-8577 JAPAN}\\
\noindent \texttt{E-mail address: kunotti@tsuda.ac.jp}


\begin{thebibliography}{00}
\bibitem{E} P.\ Etingof,
Casimirs of the Goldman Lie algebra of a closed surface,
Intern.\ Math.\ Res.\ Notices \textbf{2006}, 1-5 (2006)

\bibitem{Fuks} D.\ B.\ Fuks,
Cohomology of infinite-dimensional Lie algebras, 
Consultants Bureau, New York (1986)

\bibitem{Go} W.\ M.\ Goldman,
Invariant functions on Lie groups and
Hamiltonian flows of
surface groups representations,
Invent.\ Math.\ \textbf{85}, 263-302 (1986)

\bibitem{Gon} L.\ V.\ Goncharova,
Cohomology of Lie algebra of formal vector fields on the line,
Functional Anal.\ Appl.\ \textbf{7} (2), 6-14 (1973)

\bibitem{Ka} N.\ Kawazumi,
Cohomological aspects of Magnus expansions,
preprint, math.GT/0505497 (2005)

\bibitem{Ka2} N.\ Kawazumi,
Harmonic Magnus expansion
on the universal family of Riemann surfaces,
preprint, math.GT/0603158 (2006)


\bibitem{KK} N.\ Kawazumi and Y.\ Kuno,
The logarithms of Dehn twists, 
preprint, arXiv:1008.5017 (2010)


\bibitem{KM} N.\ Kawazumi and S.\ Morita,
The primary approximation to the cohomology of the moduli 
space of curves and cocycles for the stable
characteristic classes, 
Math.\ Res.\ Lett.\ \textbf{3}, 629--641 (1996)


\bibitem{Kon} M.\ Kontsevich,
Formal (non)-commutative symplectic geometry, in:
``The Gel'fand Mathematical Seminars, 1990-1992'',
Birkh\"auser, Boston, 173-187  (1993) 

\bibitem{Ku} Y.\ Kuno,
A combinatorial construction of symplectic expansions,
Proc.\ Amer.\ Math.\ Soc.\ \textbf{140}, 1075--1083 (2012)

\bibitem{Mas} G.\ Massuyeau,
Infinitesimal Morita homomorphisms and the
tree-level of the LMO invariant,
Bull.\ Soc.\ Math.\ France \textbf{140}, 101--161 (2012)

\bibitem{MoS} S.\ Morita, 
Structure of the mapping class groups of surfaces: a
survey and a prospect, 
Geom.\ Topol.\ Monogr.\ \textbf{2}, 349--406 (1999)

\bibitem{MoG} S.\ Morita, 
Generators for the tautological algebra of the moduli
space of curves, 
Topology \textbf{42}, 787--819 (2003)

\bibitem{W} H.\ Weyl,
The classical groups, 
Princeton University Press, 2nd.\ ed.\ 
Princeton (1953)

\end{thebibliography}
\end{document}